\documentclass[12pt,a4paper]{article}

\usepackage{graphicx}
\usepackage{amsmath,amsfonts,amsthm}

\usepackage{tikz}
\usetikzlibrary{arrows}
\usetikzlibrary{decorations.markings}

\renewcommand{\b}[1]{\boldsymbol{#1}}
\newcommand{\bfx}{\b{x}}
\newcommand{\bA}{\boldsymbol{A}}
\newcommand{\bM}{\boldsymbol{M}}
\newcommand{\bN}{\boldsymbol{N}}
\newcommand{\bn}{\b{n}}
\newcommand{\bsigma}{\boldsymbol{\sigma}}
\newcommand{\bw}{\boldsymbol{w}}
\newcommand{\bW}{\boldsymbol{W}}
\newcommand{\bWh}{\boldsymbol{W}_{\!h}}
\newcommand{\bWz}{\boldsymbol{W}_{\!\bz}}
\newcommand{\bz}{{\boldsymbol{z}}}
\newcommand{\by}{{\boldsymbol{y}}}
\newcommand{\cA}{\mathcal{A}}
\newcommand{\cB}{\mathcal{B}}
\newcommand{\cEN}{\mathcal{E}^{\mathrm{N}}}

\newcommand{\cN}{\mathcal{N}}
\newcommand{\cT}{\mathcal{T}}
\newcommand{\ddiv}{\operatorname{div}}
\newcommand{\GammaD}{{\Gamma_{\mathrm{D}}}}
\newcommand{\GammaEbz}{{\Gamma^{\mathrm{E}}_\bz}}
\newcommand{\GammaN}{{\Gamma_{\mathrm{N}}}}
\newcommand{\GammaNbz}{{\Gamma^{\mathrm{N}}_\bz}}
\newcommand{\GammaNp}{{\Gamma_{\mathrm{N}+}}}
\newcommand{\GammaNpbz}{{\Gamma^{\mathrm{N}+}_\bz}}
\newcommand{\GammaNz}{{\Gamma_{\mathrm{N}0}}}
\newcommand{\GammaNzbz}{{\Gamma^{\mathrm{N}0}_\bz}}
\newcommand{\hbu}{\hat{\b{u}}}
\newcommand{\hbv}{\hat{\b{v}}}
\newcommand{\hbw}{\hat{\b{w}}}
\newcommand{\Hdiv}[1][\Omega]{\boldsymbol{H}(\ddiv,#1)}
\newcommand{\hu}{\hat{u}}
\newcommand{\hv}{\hat{v}}
\newcommand{\hw}{\hat{w}}
\newcommand{\norm}[1]{\left\|#1\right\|}
\newcommand{\oGammaD}{\overline\Gamma_{\mathrm{D}}}
\newcommand{\oGammaN}{\overline\Gamma_{\mathrm{N}}}
\newcommand{\R}{\mathbb{R}}
\newcommand{\rmN}{{\mathrm{N}}}
\newcommand{\RT}{\boldsymbol{\mathrm{RT}}}
\newcommand{\tbsigma}{\boldsymbol{\tilde\sigma}}
\newcommand{\tbWh}{\widetilde{\boldsymbol{W}}_{\!h}}
\newcommand{\tbWz}{\widetilde{\boldsymbol{W}}_{\!\bz}}

\newtheorem{theorem}{Theorem}[section]
\newtheorem{lemma}[theorem]{Lemma}

\begin{document}

\title{Flux reconstructions in the Lehmann--Goerisch method for lower bounds on eigenvalues}
 
\author{Tom\'a\v{s} Vejchodsk\'y\\[2mm]
\parbox{\textwidth}{%\small
% % affiliations
\begin{center}
Institute of Mathematics, Czech Academy of Sciences\\
\v{Z}itn\'a 25, Praha 1, CZ-115\,67, Czech Republic\\
vejchod@math.cas.cz
\end{center}
}
}

%\date{1 December, 2016}

\maketitle

%\begin{frontmatter}

% \title{Elsevier \LaTeX\ template\tnoteref{mytitlenote}}
% \tnotetext[mytitlenote]{Fully documented templates are available in the elsarticle package on \href{http://www.ctan.org/tex-archive/macros/latex/contrib/elsarticle}{CTAN}.}

% %% Group authors per affiliation:
% % \author{Elsevier\fnref{myfootnote}}
% % \address{Radarweg 29, Amsterdam}
% % \fntext[myfootnote]{Since 1880.}
% \author{Tom\'a\v{s} Vejchodsk\'y}
% \address{Institute of Mathematics, Czech Academy of Sciences, \v{Z}itn\'a 25, CZ-115\,67 Prague 1, Czech Republic, e-mail: vejchod@math.cas.cz}
% %\ead{vejchod@math.cas.cz}

% %% or include affiliations in footnotes:
% \author[mymainaddress,mysecondaryaddress]{Elsevier Inc}
% \ead[url]{www.elsevier.com}
% 
% \author[mysecondaryaddress]{Global Customer Service\corref{mycorrespondingauthor}}
% \cortext[mycorrespondingauthor]{Corresponding author}
% \ead{support@elsevier.com}
% 
% \address[mymainaddress]{1600 John F Kennedy Boulevard, Philadelphia}
% \address[mysecondaryaddress]{360 Park Avenue South, New York}

\begin{abstract}
%This template helps you to create a properly formatted \LaTeX\ manuscript.
The standard application of the Lehmann--Goerisch method for lower bounds on eigenvalues of symmetric elliptic second-order partial differential operators relies on determination of fluxes $\tbsigma_i$ that approximate co-gradients of exact eigenfunctions scaled by corresponding eigenvalues. 
Fluxes $\tbsigma_i$ are usually computed by a global saddle point problem solved by mixed finite element methods. In this paper we propose a simpler global problem that yields fluxes $\tbsigma_i$ of the same quality. The simplified problem is smaller, it is positive definite, and any $\Hdiv$ conforming finite elements, such as Raviart--Thomas elements, can be used for its solution. 
In addition, these global problems can be split into a number of independent local problems on patches, which allows for trivial parallelization. The computational performance of these approaches is illustrated by numerical examples for Laplace and Steklov type eigenvalue problems.
These examples also show that local flux reconstructions enable to compute lower bounds on eigenvalues on considerably finer meshes than the traditional global reconstructions.
\end{abstract}


% \begin{keyword}
% 
% \MSC[2010] 
% \end{keyword}
% 
% \end{frontmatter}
%\linenumbers

\section{Introduction}

Methods for lower bounds on eigenvalues of symmetric elliptic partial differential operators attract growing attention in the last years \cite{Barrenechea2014,CanDusMadStaVoh2017,CarGed2014,CarGal2014,GruOva2009,HuHuaLin2014,HuHuaShe2015,KuzRep2013,LiLinXie2013,LiuOis2013,Liu2015,LuoLinXie:2012,YanZhaLin:2010}.
The Lehmann--Goerisch method stems from a long history of development \cite{Temple1928,Weinstein1937,Kato1949} and it is one of the most advanced methods. It is based on the Lehmann method \cite{Lehmann1949,Lehmann1950} and the $(\b{X},\cB,T)$ concept of Goerisch \cite{GoeHau1985}. Practically, this method relies on conforming approximations of eigenfunctions of interest, subsequent flux reconstructions, and an \emph{a priori} known (rough) lower bound of certain eigenvalue.
In this paper we concentrate on flux reconstructions that approximate co-gradients of approximate eigenfunctions scaled by corresponding eigenvalues.

From the computational point of view, the flux reconstruction is usually obtained by solving a global saddle point problem \cite{BehMerPluWie2000}. This problem is considerably larger than the original eigenvalue problem, its saddle point structure brings technical difficulties, and for large problems it is a bottleneck of this approach.

Therefore, we propose to reconstruct the fluxes by solving a smaller and simpler problem. The simpler problem provides the flux reconstruction of the same quality and in addition it is positive definite. Thus, it can be solved by any $\Hdiv$ conforming finite elements as opposed to the original saddle point problem, where a suitable mixed finite element method has to be employed. Despite these advantages, even the simpler problem for fluxes is considerably larger than the eigenvalue problem itself. Therefore, we utilize the idea of \cite{BraSch:2008,DolErnVoh2016,ErnVoh2013} and propose localized versions of both the saddle point and simpler problems. Localized versions are based on solving independent small local problems on patches of elements and their accuracy is competitive with global problems. The main advantage of the localized problems lies in the fact that they are independent and can be solved in parallel. Their memory requirements are low and they enable to compute lower bounds on eigenvalues for considerably finer meshes than the traditional global flux reconstructions.

The main goal of this paper is to provide the flux reconstruction procedures for a general eigenvalue problem:
find $\lambda_i > 0$ and $u_i\neq 0$ such that
\begin{alignat}{2}
\nonumber
  -\ddiv( \cA \nabla u_i ) + c u_i &= \lambda_i \beta_1 u_i &\quad &\text{in }\Omega, \\
\label{eq:EPstrong}
  (\cA \nabla u_i) \cdot \bn_\Omega + \alpha u_i &= \lambda_i \beta_2 u_i &\quad &\text{on }\GammaN, \\
\nonumber
  u_i &= 0 &\quad &\text{on }\GammaD,
\end{alignat}
where $\Omega \subset \R^d$ is an open  Lipschitz domain, $d$ a dimension, $\GammaD$ and $\GammaN$ are two relatively open components of $\partial\Omega$ such that $\oGammaD \cup \oGammaN = \partial\Omega$ and $\GammaD \cap \GammaN = \emptyset$, and $\bn_\Omega$ is the unit outward facing normal vector to the boundary $\partial\Omega$. 
Note that specific choices of parameters in problem \eqref{eq:EPstrong} yield to the standard eigenvalue problems such as the Laplace eigenvalue problem and Steklov eigenvalue problem.

However, in order to explain the main idea without technicalities, we first consider the Laplace eigenvalue problem, see Sections~\ref{se:LGLap}--\ref{se:simpleLap}. 
%The rest of this paper is organized as follows.
The following sections deal with the general eigenvalue problem. 
Section~\ref{se:EP}, in particular,  shifts the eigenvalue problem \eqref{eq:EPstrong} and briefly presents its well-posedness and finite element discretization. 
Section~\ref{se:LG} introduces the Lehmann--Goerisch method and the global mixed finite element problem for the flux reconstruction. 
Section~\ref{se:simplified} analyses the Lehmann--Goerisch method and derives the simplified global problem for the flux reconstruction. 
Section~\ref{se:local} presents local versions of these global problems and transforms them to a series of independent problems on patches of elements.
Sections~\ref{se:numex}--\ref{se:Steklov} compare the accuracy and computational performance of the global and local flux reconstructions for the Laplace and Steklov-type eigenvalue problem on a dumbbell shaped domain.
Finally, Section~\ref{se:concl} draws conclusions.

\section{The Lehmann--Goerisch method for Laplace eigenvalue problem}
\label{se:LGLap}
We first describe how to obtain lower bounds on eigenvalues by the Lehmann--Goerisch method for the special case of the Laplace eigenvalue problem.
%Since the Lehmann--Goerisch method takes advantage of a positive shift of eigenvalues, we consider the Laplace eigenvalue problem shifted by a constant $\gamma > 0$. 
We seek eigenvalues $\lambda_i > 0$ and eigenfunctions $u_i \neq 0$ such that
\begin{align}
  \label{eq:Laplace}
%  - \Delta u_i + \gamma u_i &= (\lambda_i + \gamma) u_i \quad \text{in }\Omega, 
  - \Delta u_i  &= \lambda_i u_i \quad \text{in }\Omega, 
\\ \nonumber
  u_i &= 0 \quad \text{on }\partial\Omega.
\end{align}
The weak formulation of this problem is posed in the Sobolev space $V = H^1_0(\Omega)$ consisting of $H^1(\Omega)$ functions with vanishing traces on $\partial\Omega$ and reads as follows: find eigenvalues $\lambda_i > 0$ and eigenfunctions $u_i \in V \setminus\{0\}$ such that
\begin{equation}
  \label{eq:Laplaceweak}
  %(\nabla u_i, \nabla v) + \gamma(u_i, v) = (\lambda_i+\gamma) (u_i, v) \quad \forall v \in V,
  (\nabla u_i, \nabla v) = \lambda_i (u_i, v) \quad \forall v \in V,
\end{equation}
where $(\cdot,\cdot)$ stands for the $L^2(\Omega)$ inner product.
This problem is well posed and posses a countable sequence of eigenvalues $0 < \lambda_1 \leq \lambda_2 \leq \cdots$, see e.g. \cite{BabOsb:1991,Boffi:2010}.

In order to discretize problem \eqref{eq:Laplaceweak} by the standard conforming finite element method, we consider $\Omega$ to be a polytope. We introduce a standard simplicial mesh $\cT_h$ in $\Omega$ and define the lowest-order finite element space
\begin{equation}
 \label{eq:defVh}
  V_h = \{ v_h \in V : v_h|_K \in P_1(K) \quad \forall K \in \cT_h \},
\end{equation}
where $P_1(K)$ is the space of affine functions on the simplex $K$.
The finite element approximation of problem \eqref{eq:Laplaceweak} corresponds to the finite dimensional problem of seeking eigenvalues $\Lambda_{h,i} \in \R$ and eigenfunctions $u_{h,i} \in V_h \setminus\{0\}$ such that
\begin{equation}
 \label{eq:Lapfem}
  %(\nabla u_{h,i}, \nabla v_h ) + \gamma ( u_{h,i}, v_h) = (\Lambda_{h,i}+\gamma) (u_{h,i}, v_h) \quad \forall v_h \in V_h.
  (\nabla u_{h,i}, \nabla v_h ) = \Lambda_{h,i} (u_{h,i}, v_h) \quad \forall v_h \in V_h.
\end{equation}
Discrete eigenvalues are naturally sorted in ascending order: $0 < \Lambda_{h,1} \leq \Lambda_{h,2} \leq \cdots \leq \Lambda_{h,N}$,
where $N = \operatorname{dim} V_h$.

It is well known that the order of convergence of the finite element approximation $\Lambda_{h,i}$ is quadratic \cite{BabOsb:1991,Boffi:2010} and that $\Lambda_{h,i}$ approximates $\lambda_i$ from above. The Lehmann--Goerisch method enables to compute approximations of $\lambda_i$ from below with the same order of convergence. 
The idea of this method is summarized in \cite[Theorem 2.1]{BehMerPluWie2000}. For the readers' convenience we recall this theorem here.
Note that $\bW = \Hdiv$ denotes the standard space of square integrable vector fields with square integrable divergence.
\begin{theorem}[Behnke, Mertins, Plum, Wieners]\label{th:BMPWorig}
Let $\tilde u_i \in V $, $\tbsigma_i \in \bW$, $i=1,2,\dots,n$, and $\rho >0$, $\gamma > 0$ be arbitrary.
Define matrices $\bM, \bN \in \R^{n \times n}$ with entries
\begin{align*}
  \bM_{ij} &= (\nabla \tilde u_i, \nabla \tilde u_j) + (\gamma - \rho) (\tilde u_i, \tilde u_j), \\
              %a(\tilde u_i, \tilde u_j) - \rho b(\tilde u_i, \tilde u_j), \\
  \bN_{ij} &= (\nabla \tilde u_i, \nabla \tilde u_j) + (\gamma - 2\rho) (\tilde u_i, \tilde u_j)
           + \rho^2 (\tbsigma_i,\tbsigma_j) 
           + (\rho^2/\gamma) (\tilde u_i + \ddiv\tbsigma_i,\tilde u_j + \ddiv\tbsigma_j).
%            a(\tilde u_i, \tilde u_j) - 2\rho b(\tilde u_i, \tilde u_j)
%            + \rho^2 (\cA \tbsigma_i,\tbsigma_j) 
%            + \rho^2 ( \frac{1}{c+\gamma\beta_1} [\beta_1 \tilde u_i + \ddiv(\cA \tbsigma_i)], \beta_1 \tilde u_j + \ddiv(\cA\tbsigma_j) )
%        \\  &\quad
%            + \rho^2 ( \frac{1}{\alpha+\gamma\beta_2} [\beta_2 \tilde u_i - (\cA \tbsigma_i)\cdot\bn_\Omega], \beta_2 \tilde u_j - (\cA\tbsigma_j)\cdot\bn_\Omega )_\GammaN
%            %+ \rho^2 ( [\alpha+\gamma\beta_2] \frac{\tilde u_i}{\Lambda_{h,i} + \gamma} , \frac{\tilde u_j}{\Lambda_j + \gamma} )_\GammaN
\end{align*}
Suppose that the matrix $\bN$ is positive definite and that
$$
  \mu_1 \leq \mu_2 \leq \dots \leq \mu_n
$$
are eigenvalues of the generalized eigenvalue problem 
\begin{equation}
  \label{eq:MNproblem}
  \bM \by_i = \mu_i \bN \by_i, \quad i=1,2,\dots,n.
\end{equation}
Then, for all $i$ such that $\mu_i < 0$, the interval 
$$
  [ \rho - \gamma - \rho/(1-\mu_i), \rho - \gamma)
$$
contains at least $i$ eigenvalues of the continuous problem \eqref{eq:Laplace}.
\end{theorem}

In order to use Theorem~\ref{th:BMPWorig} for obtaining guaranteed lower bounds on eigenvalues, we need to choose a positive value for the shift parameter $\gamma$ and employ an \emph{a~priori} information about the spectrum. Namely, we need to know that
$$
  \rho-\gamma \leq \lambda_L \quad\text{for some index } L \geq 2.
$$ 
Then Theorem~\ref{th:BMPWorig} provides lower bounds
\begin{equation}
\label{eq:llowinc}
\rho - \gamma - \rho/(1-\mu_i) \leq \lambda_{L - i} \quad \forall i=1,2,\dots,\min\{L-1,n\}.
\end{equation}
Thus, the \emph{a~priori} knowledge of a lower bound on at least one exact eigenvalue can be utilized to compute lower bounds on eigenvalues below it. The \emph{a~priori} known lower bound can be relatively rough, but the lower bounds \eqref{eq:llowinc} have the potential to be very accurate.

In numerical examples presented below it is sufficient to obtain the \emph{a~priori} known lower bounds by using the monotonicity principle based on a comparison with a completely solvable problem. In particular, for the Laplace eigenvalue problem in two dimensions we enclose the domain $\Omega$ into a rectangle $\mathcal{R}$. The analytically known eigenvalues for $\mathcal{R}$ are then below the corresponding eigenvalues for $\Omega$. 
In this way rough \emph{a~priori} known lower bounds for all eigenvalues up to an index of interest can be easily computed. 
If these \emph{a~priori} lower bounds are not sufficiently accurate then the homotopy approach \cite{Plum1990,Plum1991} or nonconforming finite elements \cite{CarGal2014,CarGed2014,LiuOis2013,Liu2015} are recommended.

Notice that Theorem~\ref{th:BMPWorig} holds true for arbitrary $\tilde u_i \in V$ and $\tbsigma_i \in \bW$. However, in order to achieve accurate lower bounds and especially the quadratic order of convergence, they have to be chosen such that $\tilde u_i$ approximates $u_i$ and the flux $\tbsigma_i$ approximates the scaled gradient $(\lambda_i+\gamma)^{-1} \nabla u_i$. Concerning $\tilde u_i$, it is natural to choose $\tilde u_i = u_{h,i}$. Fluxes $\tbsigma_i$ can be computed using the complementarity technique \cite{Complement:2010,systemaee:2010}, also known as dual finite elements \cite{HasHla:1976,Hla:1978,HlaKri:1984}, two energies principle \cite{Braess2013}, or complementary variational principle \cite{BehnkeGoerish1994}. Specifically, in \cite{BehMerPluWie2000} it is proposed to solve a global saddle point problem using mixed finite elements. 

In particular, we use the first order Raviart--Thomas elements and the space of piecewise affine and globally discontinuous functions. Let $\RT_1(K) = [P_1(K)]^2 \oplus \bfx P_1(K)$ be the standard local Raviart--Thomas space. Using the same triangulation $\cT_h$ as above, we define spaces
\begin{align}
%   \bWh &= \left\{ \bsigma_h \in \Hdiv : \bsigma_h|_K \in \RT_1(K) \ \forall K \in \cT_h \text{ and }
%     \bsigma_h \cdot \bn_\Omega = \frac{\Lambda_{h,i}\beta_2 - \alpha}{\Lambda_{h,i} + \gamma} u_{h,i} \text{ on }\GammaN
%   \right\},
% \\
  \label{eq:defWhLap}
  \bWh &= \left\{ \bsigma_h \in \bW : \bsigma_h|_K \in \RT_1(K) \ \forall K \in \cT_h 
  \right\},
\\  
  \label{eq:defQhLap}
  Q_h &= \{ \varphi_h \in L^2(\Omega) : \varphi_h |_K \in P_1(K) \quad\forall K \in \cT_h \}.
\end{align}
The global saddle point problem then reads: find $(\bsigma_{h,i},q_{h,i}) \in \bWh \times Q_h$ such that
\begin{alignat}{2}
  \label{eq:sigLap1}
  \left(\bsigma_{h,i}, \bw_h\right) + (q_{h,i}, \ddiv \bw_h) &= \left(\frac{\nabla u_{h,i}}{\Lambda_{h,i}+\gamma}, \bw_h \right) &\quad&\forall \bw_h \in \bWh, 
  \\
  \label{eq:sigLap2}
  (\ddiv \bsigma_{h,i}, \varphi_h )
    &=  - \left( \frac{ \Lambda_{h,i} }{\Lambda_{h,i}+\gamma} u_{h,i}, \varphi_h \right) 
  &\quad &\forall \varphi_h \in Q_h,
\end{alignat}
where $\Lambda_{h,i} \in \R$ and $u_{h,i} \in V_h$ are finite element approximations \eqref{eq:Lapfem} of the exact eigenpair.
% Fluxes $\bsigma_{h,i} \in \bWh$ computed by solving \eqref{eq:sigLap1}--\eqref{eq:sigLap2} 
% %are in $\Hdiv$ by their construction. Since this is the only requirement on fluxes in Theorem~\ref{th:BMPW}, 
% can be directly utilized as $\tbsigma_i$ in Theorem~\ref{th:BMPWorig} to provide lower bounds \eqref{eq:llowinc},
% see Lemma~\ref{le:sigOKLap} below.

\section{Simplified and local flux reconstructions for Laplace eigenvalue problem}
\label{se:simpleLap}
The traditional global saddle point problem \eqref{eq:sigLap1}--\eqref{eq:sigLap2} is not the only possibility how to compute quality fluxes.
This section presents three alternative flux reconstructions still in the context of the Laplace eigenvalue problem. % $\bsigma_{h,i} \in \bWh$ in step 3 of the above algorithm.
First we show that the global saddle point problem \eqref{eq:sigLap1}--\eqref{eq:sigLap2} can be replaced by a smaller symmetric positive definite problem by using the penalty method.
The global saddle point problem \eqref{eq:sigLap1}--\eqref{eq:sigLap2} corresponds to the constraint minimization problem: find $\bsigma_{h,i} \in \bWh$
\begin{equation}
  \label{eq:minglobcons}
   \text{minimizing } \norm{ \bsigma_{h,i} - \frac{\nabla u_{h,i}}{\Lambda_{h,i}+\gamma} }_{L^2(\Omega)}^2
   \text{ under the constraint }
   \ddiv \bsigma_{h,i} = -\frac{ \Lambda_{h,i} }{\Lambda_{h,i}+\gamma} u_{h,i}.
\end{equation}
This constraint, however, is not required by Theorem~\ref{th:BMPWorig} and its exact validity is superfluous. Therefore, we remove it and enforce it in a weaker sense by using a penalty parameter. Section~\ref{se:simplified} provides heuristic arguments for choosing the penalty parameter as $1/\gamma$. Thus, instead of the constraint minimization problem \eqref{eq:minglobcons} we propose to solve the following unconstrained minimization problem: find $\bsigma_{h,i} \in \bWh$
\begin{equation}
  \label{eq:minglobuncons}
   \text{minimizing } \norm{ \bsigma_{h,i} - \frac{\nabla u_{h,i}}{\Lambda_{h,i}+\gamma} }_{L^2(\Omega)}^2
   + \frac{1}{\gamma}\norm{\ddiv \bsigma_{h,i} + \frac{ \Lambda_{h,i} }{\Lambda_{h,i}+\gamma} u_{h,i}}_{L^2(\Omega)}^2.
\end{equation}

The Euler--Lagrange equations for this minimization problem read: 
find $\bsigma_{h,i} \in \bWh$ such that
\begin{multline}
\label{eq:sigLapB}
  \left(\bsigma_{h,i}, \bw_h \right) 
  + \frac{1}{\gamma} \left( \ddiv \bsigma_{h,i}, \ddiv \bw_h \right)
\\  
  = 
  \left( \frac{\nabla u_{h,i}}{\Lambda_{h,i}+\gamma}, \bw_h \right)
  - \frac{1}{\gamma} \left( \frac{\Lambda_{h,i}}{\Lambda_{h,i}+\gamma} u_{h,i}, \ddiv \bw_h \right)
\end{multline}
for all $\bw_h \in \bWh$.
This problem is smaller than problem \eqref{eq:sigLap1}--\eqref{eq:sigLap2} and it is positive definite. In spite of that  it is still considerably larger than the original eigenvalue problem \eqref{eq:Lapfem} in terms of degrees of freedom and its solution is is still a bottleneck for large scale computations.

Therefore, we use a partition of unity to localize these global problems and obtain quality flux reconstructions by solving small independent local problems on patches of elements. The main advantage of this localization is that these local problems can be efficiently solved in parallel.
The idea we utilize here comes from \cite{BraSch:2008} and it was worked out for example in \cite{DolErnVoh2016,ErnVoh2013} for boundary value problems.

Let $\cN_h$ denote the set of nodes in the mesh $\cT_h$ and let $\psi_\bz$ be a hat function corresponding to the node $\bz\in\cN_h$, i.e. $\psi_\bz$ is a piecewise linear and continuous function that equals to one at $\bz$ and vanishes at all other nodes of $\cT_h$.
Hat functions $\psi_\bz$ clearly form a partition of unity $\sum_{\bz\in\cN_h} \psi_\bz \equiv 1$ in $\Omega$.
Further, let $\cT_\bz = \{ K \in \cT_h : \bz \in K \}$ be the set of elements sharing vertex~$\bz\in\cN_h$. The interior of the union of all elements $K \in \cT_\bz$ is denoted by $\omega_\bz$ and called a patch. The unit outward facing normal vector to $\partial\omega_\bz$ is denoted by $\bn_\bz$. Note that $\overline\omega_\bz = \operatorname{supp} \psi_\bz$.
Furthermore, let $\GammaEbz$ be the union of those edges on the boundary $\partial\omega_\bz$ that do not contain $\bz$.
Thus, $\GammaEbz = \partial\omega_\bz$ for all interior patches, but not for the boundary patches.

In order to define the localized versions of global problems \eqref{eq:sigLap1}--\eqref{eq:sigLap2} and \eqref{eq:sigLapB}, we introduce the following spaces on patches $\omega_\bz$:
\begin{align}
%   \bWz &= \left\{ \rule{0pt}{13pt} \bsigma_\bz \in \Hdiv[\omega_\bz] : \bsigma_\bz|_K \in \RT_1(K) \ \forall K \in \cT_\bz, \right. %\quad
%  \\ &\quad \left.
%     \bsigma_\bz \cdot \bn_\bz = 0 \text{ on }\GammaEbz\cup\GammaNzbz,
%   \quad 
%     \bsigma_\bz \cdot \bn_\bz = \frac{\Lambda_{h,i}\beta_2 - \alpha}{\Lambda_{h,i} + \gamma} \Pi_E(\psi_\bz u_{h,i}) \text{ on edges } E \subset \GammaNpbz
%     %\text{ and } (\cA \bsigma_h) \cdot \bn_\Omega = \beta_2 u_{h,i} \text{ on edges} \cEzBN\cap\GammaNz
%   \right\},
% \\
\nonumber
  \bWz &= \left\{ \bsigma_\bz \in \Hdiv[\omega_\bz] : \bsigma_\bz|_K \in \RT_1(K) \ \forall K \in \cT_\bz \right.
\\ &\quad\hspace{46mm}
\nonumber
  \left. \text{ and } 
    \bsigma_\bz \cdot \bn_\bz = 0 \text{ on edges }E \subset \GammaEbz
  \right\},
\\
\label{eq:defQz}
  Q_\bz &= \{ \varphi_h \in L^2(\omega_\bz) : \varphi_h |_K \in P_1(K) \quad\forall K \in \cT_\bz \}.
\end{align}
Localization of the saddle point problem \eqref{eq:sigLap1}--\eqref{eq:sigLap2} can then be done as follows.
Compute $\bsigma_{h,i} \in \bWh$ as
\begin{equation}
\label{eq:sigLapsum}
  \bsigma_{h,i} = \sum\limits_{\bz\in\cN_h} \bsigma_{\bz,i},
\end{equation}
where each $\bsigma_{\bz,i}$ is determined by solving the following problem: 
find $(\bsigma_{\bz,i},q_{\bz,i}) \in \bWz \times Q_\bz$ such that
\begin{gather}
  \label{eq:sigLap1loc}
  \left(\bsigma_{\bz,i}, \bw_h\right)_{\omega_\bz} + (q_{\bz,i}, \ddiv \bw_h)_{\omega_\bz} = \left(\psi_\bz \frac{\nabla u_{h,i}}{\Lambda_{h,i}+\gamma}, \bw_h\right)_{\omega_\bz} 
  \quad \forall \bw_h \in \bWz, 
\\ \label{eq:sigLap2loc} %\nonumber
  (\ddiv \bsigma_{\bz,i}, \varphi_h )_{\omega_\bz}
    =  - \left( \frac{\Lambda_{h,i}}{\Lambda_{h,i}+\gamma} \psi_\bz u_{h,i}, \varphi_h\right)_{\omega_\bz} 
  + \left( \frac{\nabla \psi_\bz \cdot \nabla u_{h,i}}{\Lambda_{h,i}+\gamma}, \varphi_h\right)_{\omega_\bz}
  \quad \forall \varphi_h \in Q_\bz.
\end{gather}
Note that the last term on the right-hand side of \eqref{eq:sigLap2loc} has to be added due to solvability of this saddle point problem.
Indeed, for interior and Neumann nodes, equation \eqref{eq:sigLap2loc} tested by $\varphi_h\equiv 1$ is only consistent thanks to this term and identity \eqref{eq:Lapfem}.
Further note that summing equality \eqref{eq:sigLap2loc} over $\bz \in \cN_h$ yields the original equality \eqref{eq:sigLap2}, because the last term in \eqref{eq:sigLap2loc} vanishes.

Alternatively, we can set up local positive definite problems on patches by localizing the positive definite global problem \eqref{eq:sigLapB}.
We seek $\bsigma_{h,i} \in \bWh$
in the form \eqref{eq:sigLapsum}, where $(\bsigma_{\bz,i},q_{\bz,i}) \in \bWz \times Q_\bz$ are such that
\begin{multline}
\label{eq:sigLapBloc}
  \left(\bsigma_{\bz,i}, \bw_h\right)_{\omega_\bz} 
  + \frac{1}{\gamma}\left( \ddiv\bsigma_{\bz,i}, \ddiv \bw_h \right)_{\omega_\bz}  
  = 
  \left( \psi_\bz\frac{\nabla u_{h,i}}{\Lambda_{h,i}+\gamma}, \bw_h \right)_{\omega_\bz}
\\  
  - \frac{1}{\gamma} \left( \frac{\Lambda_{h,i}}{\Lambda_{h,i}+\gamma} \psi_\bz u_{h,i}, \ddiv\bw_h \right)_{\omega_\bz}  
  + \frac{1}{\gamma} \left( \frac{\nabla \psi_\bz \cdot \nabla u_{h,i}}{\Lambda_{h,i}+\gamma}, \ddiv\bw_h\right)_{\omega_\bz}  
\end{multline}
for all $\bw_h \in \bWz$.

It is easy to see that all presented flux reconstructions can be directly used in Theorem~\ref{th:BMPWorig} to compute lower bounds on eigenvalues \eqref{eq:llowinc}. The formal prove of this fact follows as a special case of Lemmas~\ref{le:sigOK}, \ref{le:sigBOK}, \ref{le:locsigcons}, and \ref{le:sigBlocOK} stated below.

\section{General eigenvalue problem and its discretization}
\label{se:EP}
From now on we consider the eigenvalue problem \eqref{eq:EPstrong} and generalize the ideas indicated in the previous two sections. We will provide more details and explain certain relations behind the Lehmann--Goerisch method and the proposed flux reconstructions.

Since the parameter $\gamma > 0$ plays the role of the shift, we start by formulating the shifted version of the eigenvalue problem \eqref{eq:EPstrong}:
\begin{alignat}{2}
\nonumber
  -\ddiv( \cA \nabla u_i ) + (c+\gamma\beta_1) u_i &= (\lambda_i+\gamma) \beta_1 u_i &\quad &\text{in }\Omega, \\
\label{eq:EPstrongsft}
  (\cA \nabla u_i) \cdot \bn_\Omega + (\alpha+\gamma\beta_2) u_i &= (\lambda_i+\gamma) \beta_2 u_i &\quad &\text{on }\GammaN, \\
\nonumber
  u_i &= 0 &\quad &\text{on }\GammaD.
\end{alignat}
In order to solve this problem by the conforming finite element method, we will formulate it in a weak sense.
For this purpose, we assume the diffusion matrix $\cA \in [L^\infty(\Omega)]^{d\times d}$ to be symmetric and uniformly positive definite, i.e. there exists $C>0$ such that
$$
  \b{\xi}^\top \cA(\bfx) \b{\xi} \geq C |\b{\xi}|^2 \quad\text{for all }\b{\xi} \in\R^d \text{ and for almost all } \bfx \in \Omega.
$$
This assumption implies that the inverse matrix $\cA^{-1}(x)$ exists for almost all $\bfx \in \Omega$ and that $\cA^{-1} \in L^\infty(\Omega)^{d\times d}$.
The other coefficients are $c,\beta_1 \in L^\infty(\Omega)$, $\alpha, \beta_2 \in L^\infty(\GammaN)$ and they are all assumed to be nonnegative.

We define the usual space
\begin{equation}
  \label{eq:defV}
  V = \{ v \in H^1(\Omega) : v = 0 \text{ on } \GammaD \},
\end{equation}
and we introduce bilinear forms
\begin{align}
\label{eq:blf}
  a(u,v) &= (\cA \nabla u,\nabla v) + ( [c+\gamma\beta_1] u, v) + ( [\alpha+\gamma\beta_2] u, v)_\GammaN,
%  \int_\Omega (\nabla u)^T \cA \nabla v \dx
% +\int_\Omega c u v \dx  + \int_\GammaN \alpha u v \dx[s].
\\
\label{eq:blfb}
  b(u,v) &= (\beta_1 u, v) + (\beta_2 u, v)_\GammaN,
\end{align}
where
$(\cdot,\cdot)$ stands for the $L^2(\Omega)$, and
$(\cdot,\cdot)_\GammaN$ for the $L^2(\GammaN)$ inner products.

For the form $b(\cdot,\cdot)$ we assume that at least one of the following two conditions is satisfied:
(a) $\beta_1 > 0$ on a subset of $\Omega$ of positive measure,
(b) $\beta_2 > 0$ on a subset of $\GammaN$ of positive measure.
This assumption guarantees that the eigenvalue problem does not degenerate and posses the countable infinity of eigenvalues.
Since $\gamma > 0$, the bilinear form $a(\cdot,\cdot)$ is $V$-elliptic even if $\GammaD$ is empty, $c=0$ in $\Omega$, and $\alpha=0$ on $\GammaN$. The form $a(\cdot,\cdot)$ induces a norm on $V$ denoted by $\|{\cdot}\|_a$.
% Further, we assume that that bilinear form $a(\cdot,\cdot)$ is $V$-elliptic. 
% This is the case if at least one of the following conditions is satisfied:
% (a) $c > 0$ on a subset of $\Omega$ of positive measure,
% (b) $\alpha > 0$ on a subset of $\GammaN$ of positive measure,
% (c) measure of $\GammaD$ is positive.
The form $b(\cdot,\cdot)$ induces a seminorm on $V$, in general, and we denote it by $|{\cdot}|_b$.

Under these assumptions, the weak formulation of \eqref{eq:EPstrongsft} reads: find $\lambda_i > 0$ and $u_i \in V \setminus \{0\}$ such that
\begin{equation}
\label{eq:EPweak}
  a(u_i,v) = (\lambda_i+\gamma) b(u_i,v) \quad \forall v \in V
\end{equation}
is well posed and eigenvalues form a countable sequence: $0 < \lambda_1 \leq \lambda_2 \leq \cdots$.
This follows from the standard compactness argument \cite{BabOsb:1991,Boffi:2010}, see also \cite{VejSeb2017} for this specific setting.

We discretize problem \eqref{eq:EPweak} in the same way as in \eqref{eq:Lapfem}. 
In particular, we consider the finite element space \eqref{eq:defVh}, now with $V$ given by \eqref{eq:defV},
and define approximate eigenvalues $\Lambda_{h,i} \in \R$ and eigenfunctions $u_{h,i} \in V_h\setminus\{0\}$
such that
\begin{equation}
 \label{eq:fem}
  a( u_{h,i}, v_h) = (\Lambda_{h,i}+\gamma) b(u_{h,i}, v_h) \quad \forall v_h \in V_h.
\end{equation}

\section{The Lehmann--Goerisch method for the general eigenvalue problem}
\label{se:LG}
% DESCRIPTION OF THE ORIGINAL IDEA, GLOBAL MIXED PROBLEM
In this section, we generalize the Lehmann--Goerisch method as it is described in \cite{BehMerPluWie2000} to the problem with variable coefficients \eqref{eq:EPstrong} admitting both the standard and Steklov type eigenvalue problems.
We first formulate and prove the generalization of Theorem~\ref{th:BMPWorig}, see \cite[Theorem 2.1]{BehMerPluWie2000}.

For this purpose we introduce threshold values $c_0 > 0$, $\beta_{1,0} > 0$, $\alpha_0 >0$, and $\beta_{2,0} > 0$ and define sets
\begin{align}
\label{eq:Omegap}
  \Omega_+ &= \{ \bfx \in \Omega : c(\bfx) \geq c_0 \text{ or } \beta_1(\bfx) \geq \beta_{1,0} \},
\\  
\label{eq:GammaNp}
  \GammaNp &= \{ \bfx \in \GammaN : \alpha(\bfx) \geq \alpha_0 \text{ or } \beta_2(\bfx) \geq \beta_{2,0} \}.
\end{align}
We also set $\Omega_0 = \Omega\setminus\Omega_+$ and $\GammaNz = \GammaN\setminus\GammaNp$ 
and recall that $\bW = \Hdiv$.

\begin{theorem}\label{th:BMPW}
Let $\tilde u_i \in V $, $i=1,2,\dots,n$, and $\rho >0$, $\gamma > 0$ be arbitrary.
Let $\tbsigma_i \in \bW$ be such that
\begin{equation}
\label{eq:sigmacond}
  \beta_1 \tilde u_i + \ddiv\tbsigma_i = 0 \text{ in } \Omega_0 
  \quad\text{and}\quad
  \beta_2 \tilde u_i - \tbsigma_i\cdot\bn_\Omega = 0 \text{ on } \GammaNz
\end{equation}
% $$
%   \bW_i = \{ \bw \in \Hdiv : \beta_1 \tilde u_i + \ddiv(\cA\bw) = 0 \text{ in } \Omega_0 \text{ and } \beta_2 \tilde u_i - (\cA \bw)\cdot\bn_\Omega = 0 \text{ on } \GammaNz \}.
% $$
for $i=1,2,\dots,n$.
Define matrices $\bA_0, \bA_1, \bA_2 \in \R^{n \times n}$ with entries
\begin{align}
\nonumber
  \bA_{0,ij} &= a(\tilde u_i, \tilde u_j), \quad
    \bA_{1,ij} = b(\tilde u_i, \tilde u_j), 
\\ \nonumber
  \hat\bA_{2,ij} &= \left(\cA^{-1} \tbsigma_i,\tbsigma_j\right) 
    + \left( \frac{1}{c+\gamma\beta_1} [\beta_1 \tilde u_i + \ddiv \tbsigma_i], \beta_1 \tilde u_j + \ddiv\tbsigma_j \right)_{\Omega_+}
\\ \label{eq:defhatA} 
   &\quad    
    + \left( \frac{1}{\alpha+\gamma\beta_2} [\beta_2 \tilde u_i - \tbsigma_i\cdot\bn_\Omega], \beta_2 \tilde u_j - \tbsigma_j\cdot\bn_\Omega \right)_\GammaNp
\end{align}
and matrices $\bM = \bA_0 - \rho \bA_1$, $\bN = \bA_0 - 2\rho \bA_1 + \rho^2 \hat\bA_2$.
% Define matrices $\bM, \bN \in \R^{n \times n}$ with entries
% \begin{align*}
%   \bM_{ij} &= %(\nabla \tilde u_i, \nabla \tilde u_j) + (\gamma - \rho) (\tilde u_i, \tilde u_j), \\
%               a(\tilde u_i, \tilde u_j) - \rho b(\tilde u_i, \tilde u_j), \\
%   \bN_{ij} &= %(\nabla \tilde u_i, \nabla \tilde u_j) + (\gamma - 2\rho) (\tilde u_i, \tilde u_j)
%            %+ \rho^2 (\tbsigma_i,\tbsigma_j) 
%            %+ (\rho^2/\gamma) (\tilde u_i + \ddiv\tbsigma_i,\tilde u_j + \ddiv\tbsigma_j).
%            a(\tilde u_i, \tilde u_j) - 2\rho b(\tilde u_i, \tilde u_j)
%            + \rho^2 (\cA \tbsigma_i,\tbsigma_j) 
%            + \rho^2 ( \frac{1}{c+\gamma\beta_1} [\beta_1 \tilde u_i + \ddiv(\cA \tbsigma_i)], \beta_1 \tilde u_j + \ddiv(\cA\tbsigma_j) )
%        \\  &\quad
%            + \rho^2 ( \frac{1}{\alpha+\gamma\beta_2} [\beta_2 \tilde u_i - (\cA \tbsigma_i)\cdot\bn_\Omega], \beta_2 \tilde u_j - (\cA\tbsigma_j)\cdot\bn_\Omega )_\GammaN
%            %+ \rho^2 ( [\alpha+\gamma\beta_2] \frac{\tilde u_i}{\Lambda_{h,i} + \gamma} , \frac{\tilde u_j}{\Lambda_j + \gamma} )_\GammaN
% \end{align*}
Suppose that the matrix $\bN$ is positive definite and that
$$
  \mu_1 \leq \mu_2 \leq \dots \leq \mu_n
$$
are eigenvalues of the generalized eigenvalue problem 
\begin{equation}
  \label{eq:MNproblem}
  \bM \by_i = \mu_i \bN \by_i, \quad i=1,2,\dots,n.
\end{equation}
Then, for all $i$ such that $\mu_i < 0$, the interval 
$$
  [ \rho - \gamma - \rho/(1-\mu_i), \rho - \gamma)
$$
contains at least $i$ eigenvalues of the continuous problem \eqref{eq:EPstrong}.
\end{theorem}
\begin{proof}
The proof follows from \cite[Theorem~5]{BehnkeGoerish1994}.
To verify its assumptions, we define the space $\b{X} = [L^2(\Omega)]^{d+1} \times L^2(\GammaN)$. For elements $\hbu = (\hu_1,\dots,\hu_d,\hu_{d+1},\hu_{d+2})^\top \in \b{X}$ we consider notation $\hbu = \left(\hbu^{(d)}, \hu^0, \hu^\rmN\right)^\top$, where $\hbu^{(d)} = (\hu_1,\dots,\hu_d)^\top$ is a vector with $d$ components. Using this notation, we define the bilinear form
\begin{equation}
\label{eq:defB}
  \cB(\hbu,\hbv) = \left(\cA^{-1} \hbu^{(d)}, \hbv^{(d)}\right) + \left( [c+\gamma\beta_1] \hu^0, \hv^0 \right) 
    + \left( [\alpha+\gamma\beta_2] \hu^\rmN, \hv^\rmN \right)_\GammaN
\end{equation}
on $\b{X}$.
We also define the linear operator $T : V \rightarrow \b{X}$ as
\begin{equation}
\label{eq:defT}
  T u = (\cA \nabla u, u, u|_\GammaN)^\top.
\end{equation}  
By this construction we immediately have
\begin{equation}
\label{eq:auvBTuTv}
  a(u,v) = \cB(Tu, Tv) \quad \forall u,v\in V.
\end{equation}

Now, given $\tbsigma_i \in \bW$ satisfying \eqref{eq:sigmacond}, we define $\hbw_i = \left(\hbw_i^{(d)}, \hw_i^0, \hw_i^\rmN\right)^\top \in \b{X}$ as
\begin{align}
\label{eq:hwiform}
\hbw_i^{(d)} &= \tbsigma_i, %\quad
\\ \nonumber
\hw_i^0 &= \left\{ \begin{array}{ll}
  \displaystyle\frac{\beta_1 \tilde u_i + \ddiv \tbsigma_i}{c+\gamma\beta_1} & \text{ in } \Omega_+, \\
  0 & \text{ in } \Omega_0,
  \end{array}\right.
\quad  
\hw_i^\rmN = \left\{ \begin{array}{ll}
  \displaystyle\frac{\beta_2 \tilde u_i - \tbsigma_i\cdot\bn_\Omega}{\alpha+\gamma\beta_2} & \text{ in } \GammaNp, \\
  0 & \text{ in } \GammaNz.
  \end{array}\right.
\end{align}
% $\hbw_{i,12} = \bsigma_i$, 
% $\hw_{i,3} = (c+\gamma\beta_1)^{-1} (\beta_1 \tilde u_i + \ddiv(\cA \tbsigma_i))$ in $\Omega_+$
% and $\hw_{i,3} = 0$ in $\Omega\setminus\Omega_+$,
% $\hw_{i,4} = (\alpha+\gamma\beta_2)^{-1} (\beta_2 \tilde u_i - (\cA \sigma_i)\cdot\bn_\Omega)$ in $\GammaNp$ and
% $\hw_{i,4} = 0$ in $\GammaN\setminus\GammaNp$.
Using the divergence theorem and condition \eqref{eq:sigmacond}, it is easy to verify that
\begin{equation}
  \label{eq:defhatwi}
  \cB(\hbw_i, Tv) = b(\tilde u_i, v) \quad \forall v \in V.
\end{equation}
Similarly, we easily verify that 
\begin{equation}
  \label{eq:hatAB}
  \hat\bA_{2,ij} = \cB(\hbw_i,\hbw_j) \quad\text{for } i,j=1,2,\dots,n.
\end{equation}  
Thus, all assumptions of \cite[Theorem~5]{BehnkeGoerish1994} are satisfied and the proof is finished.
\end{proof}

Theorem~\ref{th:BMPW} is used for computing lower bounds on eigenvalues by employing an \emph{a~priori} known lower bound on a certain eigenvalue as in \eqref{eq:llowinc}.
We will now present four flux reconstruction procedures in an analogy with those presented in Sections~\ref{se:LGLap}--\ref{se:simpleLap}. However, the general eigenvalue problem \eqref{eq:EPstrong} requires a more involved approach. 

For technical reasons connected with flux reconstruction, we assume coefficients $\cA$, $c$, $\beta_1$, $\alpha$, and $\beta_2$ to be piecewise constant with respect to the mesh $\cT_h$. The constant values of these coefficients will be denoted by $\cA_K$, $c_K$, $\beta_{1K}$, $\alpha_E$, and $\beta_{2E}$ for $K\in\cT_h$ and $E\in\cEN_h$, where $\cEN_h$ stands for the set of all edges in $\cT_h$ lying on $\GammaN$. Consequently, the natural choices of the threshold values in \eqref{eq:Omegap} and \eqref{eq:GammaNp} are $c_0 = \min\{c_K > 0,\ K\in\cT_h\}$, $\beta_{1,0}=\min\{\beta_{1K} > 0,\ K\in\cT_h\}$, $\alpha_0=\min\{\alpha_E > 0,\ E\in\cEN_h\}$,
$\beta_{2,0}=\min\{\beta_{2E} > 0,\ E\in\cEN_h\}$ and the set $\Omega_0$ then consists of those elements $K\in\cT_h$ where both $c_K$ and $\beta_{1K}$ vanish. Similarly, the set $\GammaNz$ consists of those edges $E\in\cEN_h$ where both $\alpha_E$ and $\beta_{2E}$ vanish.

In order to generalize the global saddle point problem \eqref{eq:sigLap1}--\eqref{eq:sigLap2}, we need to enforce suitable values for the normal components of fluxes on the Neumann boundary. 
Therefore, we define spaces
\begin{align*}
  \bWh &= \left\{ \bsigma_h \in \Hdiv : \bsigma_h|_K \in \RT_1(K) \ \forall K \in \cT_h \text{ and }
    \bsigma_h \cdot \bn_\Omega = \frac{\Lambda_{h,i}\beta_2 - \alpha}{\Lambda_{h,i} + \gamma} u_{h,i} \text{ on }\GammaN
%  \\   
%     \text{ and } (\cA \bsigma_h) \cdot \bn_\Omega = 0 \text{ on } \GammaNz
  \right\},
\\
  \bWh^0 &= \left\{ \bsigma_h \in \Hdiv : \bsigma_h|_K \in \RT_1(K) \ \forall K \in \cT_h \text{ and } 
    \bsigma_h \cdot \bn_\Omega = 0 \text{ on }\GammaN
  \right\},
% \\  
%   Q_h &= \{ \varphi_h \in L^2(\Omega) : \varphi_h |_K \in P_1(K) \quad\forall K \in \cT_h \}.
\end{align*}
Notice the updated definition of the space $\bWh$ in comparison with \eqref{eq:defWhLap}. The space $Q_h$ will be used in the same form as in \eqref{eq:defQhLap}.

The global saddle point problem for the general eigenvalue problem then reads: find $(\bsigma_{h,i},q_{h,i}) \in \bWh \times Q_h$ such that
\begin{alignat}{2}
  \label{eq:sig1}
  \left(\cA^{-1} \bsigma_{h,i}, \bw_h\right) + (q_{h,i}, \ddiv \bw_h) &= \left(\frac{\nabla u_{h,i}}{\Lambda_{h,i}+\gamma}, \bw_h \right) &\quad&\forall \bw_h \in \bWh^0, 
  \\
  \label{eq:sig2}
  (\ddiv \bsigma_{h,i}, \varphi_h )
    &=  \left( \frac{c - \Lambda_{h,i} \beta_1}{\Lambda_{h,i}+\gamma} u_{h,i}, \varphi_h \right) 
      %- ( \beta_1 u_{h,i}, \varphi_h)_{\Omega\setminus\Omega_+} 
  &\quad &\forall \varphi_h \in Q_h,
\end{alignat}
where $\Lambda_{h,i} \in \R$ and $u_{h,i} \in V_h$ are finite element approximations \eqref{eq:fem} of the exact eigenpair.
The following lemma verifies that this flux reconstruction can be used in Theorem~\ref{th:BMPW} to compute lower bounds on eigenvalues as in \eqref{eq:llowinc}.
\begin{lemma}
\label{le:sigOK}
The flux $\bsigma_{h,i} \in \bWh$ computed by \eqref{eq:sig1}--\eqref{eq:sig2} satisfies all assumptions of Theorem~\ref{th:BMPW}.
\end{lemma}
\begin{proof}
The fact that $\bsigma_{h,i} \in \Hdiv$ is immediate from the construction.
The first condition in \eqref{eq:sigmacond} is included in the constraint \eqref{eq:sig2} on $\ddiv\bsigma_{h,i}$,
because piecewise constant coefficients $c$ and $\beta_1$ vanish in $\Omega_0$ and both $\ddiv \bsigma_{h,i}$ and $(c - \Lambda_{h,i} \beta_1)(\Lambda_{h,i}+\gamma)^{-1} u_{h,i}$ lie in $Q_h$. 
%consisting of functions being piecewise in $P_1(K)$. 
%The second condition in \eqref{eq:sigmacond} is satisfied due to the definition of the space $\tbWh$.
The second condition in \eqref{eq:sigmacond} is satisfied due to the choice of boundary conditions in $\bWh$ and the fact that  piecewise constant $\alpha$ and $\beta_2$ vanish in $\GammaNz$.
% Note that since we assume and since $c=\beta_1=0$ in $\Omega_0$, the solution $\bsigma_{h,i}$ of problem \eqref{eq:sig1}--\eqref{eq:sig2} satisfies conditions \eqref{eq:sigmacond}. This follows from identity \eqref{eq:sig2}, because 
\end{proof}

\section{Derivation of the simplified flux reconstruction}
\label{se:simplified}

The global saddle point problem \eqref{eq:sig1}--\eqref{eq:sig2} is a direct analogy of problem \eqref{eq:sigLap1}--\eqref{eq:sigLap2}, see also \cite{BehMerPluWie2000}.
In order to derive its simplified version, we will first analyse the Lehmann--Goerisch method.

% As we already mentioned, the saddle point problem \eqref{eq:sig1}--\eqref{eq:sig2} is a bottleneck of this approach especially for large-scale computations. In this section, we propose an alternative to this saddle point problem. The proposed problem is still global, but smaller and of a simpler positive definite structure. 

The Lehmann--Goerisch method stems from the Lehmann method \cite{Lehmann1949,Lehmann1950}. The original Lehmann method can be formulated as in Theorem~\ref{th:BMPW} up to one difference: matrix $\hat\bA_2$ has to be replaced by matrix $\bA_2$ defined by
$$
  \bA_{2,ij} = a(w_i,w_j), \quad i,j=1,2,\dots,n,
$$
where $w_i \in V$ is the unique function satisfying 
\begin{equation}
  \label{eq:defwi}
  a(w_i,v) = b(\tilde u_i,v) \quad \forall v \in V.
\end{equation}
Matrix $\bA_2$ is optimal in the context of Theorem~\ref{th:BMPW}, but it is not computable in practice, because functions $w_i$ are in general unknown. The $(\b{X},\cB,T)$ concept of Goerisch (as we use it the proof of Theorem~\ref{th:BMPW}) replaces $\bA_2$ by a computable matrix $\hat\bA_2$.  Thus, the idea is to construct matrix $\hat\bA_2$ as close as possible to the optimal matrix $\bA_2$.

Matrix $\hat\bA_2$ is a good approximation of $\bA_2$ if $\hbw_i$ are good approximations of $Tw_i$ for all $i=1,2,\dots,n$, because
by \eqref{eq:auvBTuTv} and \eqref{eq:hatAB}, we have $\bA_{2,ij} = a(w_i,w_j) = \cB(Tw_i,Tw_j)$ and $\hat\bA_{2,ij} = \cB(\hbw_i,\hbw_j)$.
In order to estimate the difference $Tw_i - \hbw_i$, we utilize the complementarity technique. 
%\cite{Vejch,...}, \cite{Krizek,Hlavacek}, known also as two-energy principle \cite{Braess}, or complementary variational principle \cite{BehGoe1994}.

First of all, we notice that definitions \eqref{eq:auvBTuTv}, \eqref{eq:defhatwi}, and \eqref{eq:defwi} imply
$$
  \cB(Tw_i - \hbw_i, Tv) = 0 \quad \forall v \in V.
$$
Thus, we immediately obtain the Pythagorean identity
\begin{equation}
\label{eq:pythagoras}
%  | Tw_i - \hbw_i|_\cB^2 + |Tw_i - (\tilde\lambda_i+\gamma)^{-1} T\tilde u_i|_\cB^2 = |\hbw_i - (\tilde\lambda_i+\gamma)^{-1}T \tilde u_i|_\cB^2,
  | Tw_i - \hbw_i|_\cB^2 + |Tw_i - Tz|_\cB^2 = |\hbw_i - Tz|_\cB^2
  \quad\forall z\in V,
\end{equation}
where $|\b{v}|_\cB = \cB(\b{v},\b{v})^{1/2}$ denotes the seminorm induced by $\cB$ on $\b{X}$.
Here, we use the following observation.
If the pair $\tilde\lambda_i,\tilde u_i$ is a good approximation of the exact eigenpair $\lambda_i,u_i$ then 
$a(w_i,v) = b(\tilde u_i,v) \approx (\tilde\lambda_i+\gamma)^{-1} a(\tilde u_i,v)$ for all $v\in V$
and we observe that $z = (\tilde\lambda_i+\gamma)^{-1} \tilde u_i$ is a good approximation of $w_i$.
Thus, using this choice of $z$ in \eqref{eq:pythagoras}, we have 
the term $|Tw_i - Tz|_\cB = |Tw_i - (\tilde\lambda_i+\gamma)^{-1} T\tilde u_i|_\cB = \| w_i - (\tilde\lambda_i+\gamma)^{-1} \tilde u_i\|_a$ sufficiently small.
Consequently, minimizing $|\hbw_i - T z|_\cB$ we also minimize $|Tw_i - \hbw_i|_\cB$.
This motivates us to seek suitable $\hbw_i$ that minimizes the quadratic functional
\begin{equation}
\label{eq:funcabs}
  |\hbw_i - (\tilde\lambda_i+\gamma)^{-1} T \tilde u_i|_\cB^2.
\end{equation}

% Thus, seeking $\hbw_i$ such that $|\hbw_i - T z|_\cB$ is as small as possible we achieve small values of
% $|Tw_i - \hbw_i|_\cB$. In cases when $\cB$ induces a norm on $\b{X}$, for example in the case of Laplace eigenvalue problem (i.e. $\beta_1 \equiv 1$ and $\GammaN = \emptyset$), we immediately conclude that $Tw_i$ is well approximated by $\hbw_i$. In cases when $\cB$ induces just a seminorm, for example in the case of Steklov eigenvalue problem (i.e. $c\equiv\beta_1=0$, $\alpha\equiv 0$, $\beta_2\equiv1$, $\GammaN = \partial\Omega$), we use this argument as a heuristics that motivates minimization of $|\hbw_i - (\tilde\lambda_i+\gamma)^{-1}T \tilde u_i|_\cB$.

Using specific forms \eqref{eq:defB}, \eqref{eq:defT}, and \eqref{eq:hwiform} of bilinear form $\cB$, operator $T$, and vector $\hbw_i$, respectively,
using approximations $\tilde\lambda_i = \Lambda_{h,i}$, $\tilde u_i = u_{h,i}$, 
and taking advantage of the fact that piecewise constant $c$, $\beta_1$ vanish in $\Omega_0$ and 
$\alpha$, $\beta_2$ vanish on $\GammaNz$,
the quadratic functional \eqref{eq:funcabs} admits the following form:
\begin{multline}
\label{eq:functional}
%  |\hbw_i - (\tilde\lambda_i+\gamma)^{-1}T \tilde u_i|_\cB^2 = 
  \norm{\frac{\cA^{1/2}\nabla u_{h,i}}{\Lambda_{h,i}+\gamma} - \cA^{-1/2}\tbsigma_i }_0^2
  +\norm{\frac{1}{(c+\gamma\beta_1)^{1/2}} \left( \frac{\Lambda_{h,i}\beta_1 - c}{\Lambda_{h,i}+\gamma} u_{h,i} + \ddiv\tbsigma_i \right)}_{0,\Omega_+}^2
\\  
  +\norm{\frac{1}{(\alpha+\gamma\beta_2)^{1/2}} \left( \frac{\Lambda_{h,i} \beta_2 - \alpha}{\Lambda_{h,i}+\gamma} u_{h,i} - \tbsigma_i \cdot \bn_\Omega \right)}_{0,\GammaNp}^2
%  + \norm{\frac{(c+\gamma\beta_1)^{1/2}}{\Lambda_{h,i}+\gamma} u_{h,i} }_{\Omega_0}
%  + \norm{\frac{(\alpha+\gamma\beta_2)^{1/2}}{\Lambda_{h,i}+\gamma} u_{h,i} }_\GammaNz.
\end{multline}
Notice that in the special case of the Laplace eigenvalue problem, this functional coincides with the one in \eqref{eq:minglobuncons}.

The goal is to minimize this functional over a suitable finite dimensional subspace,
namely over the first-order Raviart--Thomas space.
Defining
% \begin{multline*}
%   \tbWh = \{ \bsigma_h \in \Hdiv : \bsigma_h|_K \in \RT_1(K) \ \forall K \in \cT_h, \quad
%  \\   
%     \text{ and } (\cA \bsigma_h) \cdot \bn_\Omega = \beta_2 u_{h,i} \text{ on } \GammaN\setminus\GammaNp
%   \}
% \end{multline*}
% and 
$$
  \tbWh = \{ \bsigma_h \in \Hdiv : \bsigma_h|_K \in \RT_1(K) \ \forall K \in \cT_h \text{ and } 
    \bsigma_h \cdot \bn_\Omega = 0 \text{ on }\GammaNz
  \},
$$
we find out that the minimizer $\bsigma_{h,i} \in \tbWh$ of \eqref{eq:functional}
under the constraints
\begin{equation}
\label{eq:constraints}
\ddiv \bsigma_{h,i} = 
\frac{c - \Lambda_{h,i}\beta_1}{\Lambda_{h,i}+\gamma} u_{h,i} \text{ in }\Omega
\quad\text{and}\quad
\bsigma_{h,i} \cdot \bn_\Omega = \frac{\Lambda_{h,i}\beta_2 - \alpha}{\Lambda_{h,i}+\gamma} u_{h,i} 
\text{ on }\GammaNp
\end{equation}
% 
% $\ddiv(\cA\sigma_h) = -\beta_1 u_{h,i}$ in $\Omega\setminus\Omega_+$,
% $\frac{\beta_1\Lambda_{h,i} - c}{\Lambda_{h,i}+\gamma} u_{h,i} + \ddiv(\cA\bsigma_{h,i}) = 0$ in $\Omega_+$
% and $ \frac{\beta_2\Lambda_{h,i} - \alpha}{\Lambda_{h,i}+\gamma} u_{h,i} - (\cA\bsigma_{h,i})\cdot \bn_\Omega = 0$ on $\GammaNp$
%
solves 
the saddle point problem \eqref{eq:sig1}--\eqref{eq:sig2}.
Notice that equalities \eqref{eq:sig1}--\eqref{eq:sig2} are the Euler--Lagrange equations corresponding to this constraint minimization problem.
We also note that $\bWh \subset \tbWh$, because $\alpha$ and $\beta_2$ vanish on $\GammaNz$.

The important observation is that constraints \eqref{eq:constraints} are not necessary and we can minimize the functional \eqref{eq:functional} over $\bsigma_{h,i} \in \tbWh$ with the only constraint dictated by conditions \eqref{eq:sigmacond}. The corresponding minimizer $(\bsigma_{h,i},q_h) \in \tbWh \times \widetilde Q_h$ solves the Euler--Lagrange equations
\begin{multline}
\label{eq:sigB1}
  \left(\cA^{-1} \bsigma_{h,i}, \bw_h\right) 
  + \left( \frac{\ddiv\bsigma_{h,i}}{c+\gamma\beta_1}, \ddiv \bw_h \right)_{\Omega_+}
  + \left( \frac{\bsigma_{h,i}\cdot\bn_\Omega}{\alpha+\gamma\beta_2} , \bw_h\cdot\bn_\Omega \right)_{\GammaNp}
\\
  + (q_h, \ddiv\bw_h)_{\Omega_0}
  = 
  \left( \frac{\nabla u_{h,i}}{\Lambda_{h,i}+\gamma}, \bw_h \right)
  - \left( \frac{(\Lambda_{h,i}\beta_1 - c)u_{h,i}}{(c+\gamma\beta_1)(\Lambda_{h,i}+\gamma)} , \ddiv \bw_h \right)_{\Omega_+}
\\  
  + \left( \frac{(\Lambda_{h,i}\beta_2 - \alpha) u_{h,i}}{(\alpha+\gamma\beta_2)(\Lambda_{h,i}+\gamma)} ,  \bw_h \cdot\bn_\Omega \right)_{\GammaNp}
\end{multline}
for all $\bw_h \in \tbWh$ and
\begin{equation}
\label{eq:sigB2}
  (\ddiv \bsigma_{h,i}, \varphi_h )_{\Omega_0}
    = 0
%      - ( \beta_1 u_{h,i}, \varphi_h)_{\Omega_0} 
  \quad \forall \varphi_h \in \widetilde Q_h,  
\end{equation}
where
$$
  \widetilde Q_h = \{ q_h \in L^2(\Omega_0) : q_h |_K \in P_1(K) \quad\forall K \in \cT_h, K \subset\overline\Omega_0\}.
$$

%It is easy to verify that $\bsigma_{h,i} \in \tbWh$ computed by solving problem \eqref{eq:sigB1}--\eqref{eq:sigB2} satisfies conditions \eqref{eq:sigmacond}.
The following lemma shows that this flux reconstruction can be immediately used in the Lehmann--Goerisch method for lower bounds on eigenvalues.
\begin{lemma}
\label{le:sigBOK}
Flux reconstruction $\bsigma_{h,i} \in \tbWh$ computed by solving problem \eqref{eq:sigB1}--\eqref{eq:sigB2} satisfies all assumptions of 
Theorem~\ref{th:BMPW}.
\end{lemma}
\begin{proof}
The definition of $\tbWh$ immediately implies that $\bsigma_{h,i} \in \Hdiv$.
Equation \eqref{eq:sigB2} guarantees the validity of the first condition in \eqref{eq:sigmacond}, because the piecewise constant $\beta_1$ vanishes in $\Omega_0$ and $\ddiv\bsigma_{h_i}|_{\Omega_0}$ lies in $\widetilde Q_h$.
The second condition in \eqref{eq:sigmacond} is satisfied due to the choice of boundary conditions in $\tbWh$ and the fact that the piecewise constant $\beta_2$ vanishes in $\GammaNz$.
\end{proof}

Euler--Lagrange equations \eqref{eq:sigB1}--\eqref{eq:sigB2} are especially useful if 
\begin{equation}
\label{eq:cbeta1pos}
  \text{either}\quad c_K > 0 \quad\text{or}\quad \beta_{1K} > 0 \quad\text{or both hold for all }K\in\cT_h.
\end{equation}
In this case the domain $\Omega_0$ is empty, $\Omega_+ = \Omega$, and the saddle point problem \eqref{eq:sigB1}--\eqref{eq:sigB2} reduces to
a positive definite problem of finding $\bsigma_{h,i} \in \tbWh$ such that
\begin{multline}
\label{eq:sigC}
  \left(\cA^{-1} \bsigma_{h,i}, \bw_h \right) 
  + \left( \frac{\ddiv \bsigma_{h,i}}{c+\gamma\beta_1} , \ddiv \bw_h \right)
  + \left( \frac{\bsigma_{h,i}\cdot\bn_\Omega}{\alpha+\gamma\beta_2}, \bw_h\cdot\bn_\Omega \right)_{\GammaNp}
\\  
  = 
  \left( \frac{\nabla u_{h,i}}{\Lambda_{h,i}+\gamma}, \bw_h \right)
  - \left( \frac{(\Lambda_{h,i}\beta_1 - c)u_{h,i}}{(c+\gamma\beta_1)(\Lambda_{h,i}+\gamma)} , \ddiv \bw_h \right)
\\  
  + \left( \frac{(\Lambda_{h,i}\beta_2 - \alpha) u_{h,i}\cdot\bn_\Omega}{(\alpha+\gamma\beta_2)(\Lambda_{h,i}+\gamma)} ,  \bw_h\cdot\bn_\Omega \right)_{\GammaNp}
\end{multline}
for all $\bw_h \in \tbWh$.
Notice that this problem simplifies to \eqref{eq:sigLapB} in the special case of the Laplace eigenvalue problem.
Further notice that fluxes computed by \eqref{eq:sigC} satisfy all assumptions of Theorem~\ref{th:BMPW} by Lemma~\ref{le:sigBOK}.

% TO DO: FORMULATE THIS AS A LEMMA WITH PROOF:
% Assumption \eqref{eq:cbeta1pos} yields $\Omega_0 = \emptyset$, $\Omega_+ = \Omega$ and, thus, the first conditions in \eqref{eq:sigmacond} is irrelevant. The second condition in \eqref{eq:sigmacond} is satisfied due to the choice of boundary conditions, see the definition of $\tbWh$.

% .............
% 
% In particular, we have 
% $$
%   | Tw_i - \hbw_i|_\cB \leq |\hbw_i - T \tilde u_i|_\cB
% $$
% 
% 
% 
% 
% .............
% 
% \begin{equation}
% \label{eq:wiapprox}
%   \hbw_i \approx Tw_i.
% \end{equation}
% 
% Further, the following lemma shows that $w_i$ is an approximation of $\Lambda_{h,i}^{-1} u_{h,i}$.
% \begin{lemma}
% \label{le:piwi}
% Let $\Pi: V \rightarrow V_h$ be the elliptic projection: ...
% Then
% $$
%   \Pi w_i = \Lambda_{h,i}^{-1} u_{h,i}.
% $$
% \end{lemma}
% \begin{proof}
% \end{proof}
% 
% Thus combining the specific form of $\hbw_i$, see \eqref{???}, relation \eqref{eq:wiapprox} and approximation $w_i \approx \Lambda_{h,i}^{-1} u_{h,i}$ justified by Lemma~\ref{le:piwi}, we have
% $$
%   \hbw_i = \left( \begin{array}{c}
%     \tbsigma_i \\
%     {}[c+\gamma\beta_1]^{-} [\beta_1 \tilde u_i + \ddiv(\cA \tbsigma_i)] \\
%     ... \\
%   \end{array}\right)
%   \approx
%   \left( \begin{array}{c}
%     \nabla w_i \\
%     w_i \\
%     w_i|_\GammaN \\
%   \end{array}\right)
%   \approx
%   \frac{1}{\Lambda_{h,i} + \gamma}
%   \left( \begin{array}{c}
%     \nabla \tilde u_i \\
%     \tilde u_i \\
%     \tilde u_i|_\GammaN \\
%   \end{array}\right)
% $$

\section{Localization of global problems for the general eigenvalue problem}
\label{se:local}
%PARTITION OF UNITY AND DEFINITION OF LOCAL PROBLEMS
Global problems \eqref{eq:sig1}--\eqref{eq:sig2}, \eqref{eq:sigB1}--\eqref{eq:sigB2}, and \eqref{eq:sigC}
for fluxes $\bsigma_{h,i}$ are all considerably larger than the original eigenvalue problem \eqref{eq:fem} in terms of degrees of freedom.
Thus, solving any of these problems is the most expensive part of the computation of lower bounds, especially in terms of the computer memory. Therefore, we localize these global problems as in Section~\ref{se:simpleLap}.
We recall that this idea was developed in \cite{BraSch:2008,DolErnVoh2016,ErnVoh2013} for boundary value problems and enables to reconstruct the flux by solving a series of small independent problems.

We use the same partition of unity as in Section~\ref{se:simpleLap}. We recall hat functions $\psi_\bz$, patches of elements $\cT_\bz$ and $\omega_\bz$, and the notation $\GammaEbz$ for the union of those edges on the boundary $\partial\omega_\bz$ that do not contain $\bz$.
In addition, we introduce sets $\GammaNpbz$ and $\GammaNzbz$ as unions of edges $E\in\cEN_h$
lying either on $\GammaNp\cap\partial\omega_\bz$ or $\GammaNz\cap\partial\omega_\bz$, respectively, 
and having an end point at $\bz$.
We also set $\GammaNbz = \GammaNpbz \cup \GammaNzbz$.

% We denote by $\cEzI$??? IS IT NEEDED ??? the set of interior edges in the patch $\omega_\bz$,
% by $\cEzBE$ the set of those edges on the boundary $\partial\omega_\bz$ that do not contain $\bz$,
% and by $\cEzBD$, $\cEzBNp$, and $\cEzBNz$ the sets of edges on the boundary $\partial\omega_\bz$ with an end point at $\bz$ lying either on $\GammaD$, $\GammaNp$, or $\GammaNz$, respectively. Note that sets $\cEzBD$, $\cEzBNp$, and $\cEzBNz$ can be nonempty only if $\bz$ lies on $\partial\Omega$, i.e. for boundary patches. We also put $\cEzBN = \cEzBNp \cup \cEzBNz$.

Similarly as for global problems, we update the definition of spaces localized to patches $\omega_\bz$:
%In order to define the localized versions of global problems \eqref{eq:sig1}--\eqref{eq:sig2}, \eqref{eq:sigB1}--\eqref{eq:sigB2}, and \eqref{eq:sigC}, we introduce the following spaces on patches $\omega_\bz$:
\begin{align*}
  \bWz &= \left\{ \rule{0pt}{13pt} \bsigma_\bz \in \Hdiv[\omega_\bz] : \bsigma_\bz|_K \in \RT_1(K) \ \forall K \in \cT_\bz, \right. %\quad
 \\ &\quad \left.
    \bsigma_\bz \cdot \bn_\bz = 0 \text{ on }\GammaEbz\cup\GammaNzbz,
  \quad 
    \bsigma_\bz \cdot \bn_\bz = \frac{\Lambda_{h,i}\beta_2 - \alpha}{\Lambda_{h,i} + \gamma} \Pi_E(\psi_\bz u_{h,i}) \text{ on edges } E \subset \GammaNpbz
    %\text{ and } (\cA \bsigma_h) \cdot \bn_\Omega = \beta_2 u_{h,i} \text{ on edges} \cEzBN\cap\GammaNz
  \right\},
\\
  \bWz^0 &= \left\{ \bsigma_\bz \in \Hdiv[\omega_\bz] : \bsigma_\bz|_K \in \RT_1(K) \ \forall K \in \cT_\bz %\right.
%\\ &\quad\hspace{60mm} \left. 
  \text{ and } \bsigma_\bz \cdot \bn_\bz = 0 \text{ on } \GammaEbz\cup\GammaNbz
  \right\},
% \\
%   Q_\bz &= \{ \varphi_h \in L^2(\omega_\bz) : \varphi_h |_K \in P_1(K) \quad\forall K \in \cT_\bz \},
\end{align*}
where $\Pi_E: L^2(E) \mapsto P_1(E)$ is the $L^2$ orthogonal projection on edges $E \subset \GammaNpbz$.
Note that the space $Q_\bz$ remains the same as in \eqref{eq:defQz}.
Localization of the saddle point problem \eqref{eq:sig1}--\eqref{eq:sig2} generalizes the case of Laplace eigenvalue problem, see \eqref{eq:sigLap1loc}--\eqref{eq:sigLap2loc}.
Fluxes $\bsigma_{h,i} \in \bWh$ are computed as
\begin{equation}
\label{eq:sigsum}
  \bsigma_{h,i} = \sum\limits_{\bz\in\cN_h} \bsigma_{\bz,i},
\end{equation}
where $\bsigma_{\bz,i}$ are determined by solving the following problem: 
find $(\bsigma_{\bz,i},q_{\bz,i}) \in \bWz \times Q_\bz$ such that
\begin{gather}
  \label{eq:sig1loc}
  \left(\cA^{-1}\bsigma_{\bz,i}, \bw_h\right)_{\omega_\bz} + (q_{\bz,i}, \ddiv \bw_h)_{\omega_\bz} = \left(\psi_\bz \frac{\nabla u_{h,i}}{\Lambda_{h,i}+\gamma}, \bw_h\right)_{\omega_\bz} 
  \quad \forall \bw_h \in \bWz^0, 
\\ \label{eq:sig2loc} %\nonumber
  (\ddiv \bsigma_{\bz,i}, \varphi_h )_{\omega_\bz}
    =  \left( \frac{c - \Lambda_{h,i} \beta_1}{\Lambda_{h,i}+\gamma} \psi_\bz u_{h,i}, \varphi_h\right)_{\omega_\bz} 
  + \left( \frac{(\cA \nabla \psi_\bz) \cdot \nabla u_{h,i}}{\Lambda_{h,i}+\gamma}, \varphi_h\right)_{\omega_\bz}
  \quad \forall \varphi_h \in Q_\bz.
\end{gather}
As in the case of local problems \eqref{eq:sigLap1loc}--\eqref{eq:sigLap2loc} the consistency of equation \eqref{eq:sig2loc} for interior and Neumann nodes follows from identity \eqref{eq:fem}.
% Note that the last term on the right-hand side of \eqref{eq:sig2loc} has to be added due to solvability of this mixed problem.
% Indeed, for interior and Neumann nodes, equation \eqref{eq:sig2loc} tested by $\varphi_h\equiv 1$ is only consistent thanks to this term and identity \eqref{eq:fem}.
Interestingly, the following lemma shows that the local flux reconstruction $\bsigma_{h,i}$ given by \eqref{eq:sigsum} and \eqref{eq:sig1loc}--\eqref{eq:sig2loc} satisfies the same constraints as the original flux reconstruction computed by solving \eqref{eq:sig1}--\eqref{eq:sig2}. 
%Consequently it satisfies conditions \eqref{eq:sigmacond}.

\begin{lemma}
\label{le:locsigcons}
Let $\bsigma_{\bz,i} \in \bWz$ be solutions of problems \eqref{eq:sig1loc}--\eqref{eq:sig2loc} for all $\bz\in\cN_h$ and let $\bsigma_{h,i}$ be given by \eqref{eq:sigsum}.
Then $\bsigma_{h,i} \in \tbWh$ and it satisfies constraints \eqref{eq:constraints}.
Consequently, it satisfies all assumptions of Theorem~\ref{th:BMPW}.
\end{lemma}
\begin{proof}
Since $\bsigma_{\bz,i} \in \bWz$ have zero normal components on edges $E\subset\GammaEbz$, it can be extended by zero to entire $\Omega$ and the extension lies in $\Hdiv$. Thus, by \eqref{eq:sigsum} we conclude that $\bsigma_{h,i} \in \Hdiv$.

In order to prove the first constraint in \eqref{eq:constraints}, we set 
$$
  r_h = \ddiv \bsigma_{h,i} - \frac{c - \Lambda_{h,i} \beta_1}{\Lambda_{h,i}+\gamma} u_{h,i}
$$
and prove that $r_h = 0$.
Notice that $r_h|_K \in P_1(K)$ for all $K\in \cT_h$, because coefficients $c$ and $\beta_1$ are piecewise constant. Thus, $r_h|_{\omega_\bz} \in Q_\bz$ for all $\bz \in \cN_h$.
Using the partition of unity $\sum_{\bz\in\cN_h} \psi_\bz \equiv 1$ and \eqref{eq:sig2loc}, we obtain
\begin{align*}
  \norm{r_h}_{L^2(\Omega)}^2 &=
  \sum_{\bz\in\cN_h} \left( 
     \ddiv \bsigma_{\bz,i} - \frac{c - \Lambda_{h,i} \beta_1}{\Lambda_{h,i}+\gamma} \psi_\bz u_{h,i}
      - \frac{(\cA \nabla \psi_\bz) \cdot \nabla u_{h,i}}{\Lambda_{h,i}+\gamma},
    r_h \right)_{\omega_\bz}
 &= 0.
\end{align*}

To prove that normal components of $\bsigma_{h,i}$ satisfy the second constraint in \eqref{eq:constraints},
we introduce the set $\cN_E$ of the two end points of the edge $E\in\cEN_h$ and use boundary conditions specified in the definition of $\bWz$. On every edge $E \subset \GammaNp$ we have
$$
  \bsigma_{h,i} \cdot \bn_\Omega
  = \sum_{\bz \in \cN_E} \bsigma_{\bz,i} \cdot \bn_\bz
  = \frac{\Lambda_{h,i}\beta_2 - \alpha}{\Lambda_{h,i} + \gamma} \Pi_E\left(\sum_{\bz \in \cN_E} \psi_\bz u_{h,i}\right) 
  = \frac{\Lambda_{h,i}\beta_2 - \alpha}{\Lambda_{h,i} + \gamma} u_{h,i},
$$
where we use properties of the projection $\Pi_E$ and the fact that $\sum_{\bz\in\cN_E} \psi_\bz = 1$ on the edge $E$.
Similarly, it is easy to see that $\bsigma_{h,i} \cdot \bn_\Omega = 0$ on $\GammaNz$. 

Thus, $\bsigma_{h,i}$ lies in $\tbWh$ and satisfies both constraints in \eqref{eq:constraints}.
Since $c$, $\beta_1$ and $\alpha$, $\beta_2$ are piecewise constant and vanish in $\Omega_0$ and $\GammaNz$, respectively, we immediately see that conditions \eqref{eq:sigmacond} in Theorem~\ref{th:BMPW} are satisfied.
\end{proof}

%In a similar way 
%as we constructed local problems \eqref{eq:sig1loc}--\eqref{eq:sig2loc}, 
To localize the global saddle point problem \eqref{eq:sigB1}--\eqref{eq:sigB2}, we have to remove the prescribed values of normal components of reconstructed fluxes on $\GammaNpbz$.
For that purpose, we introduce spaces
\begin{align*}
\tbWz &= \{ \bw_h \in \Hdiv[\omega_\bz] : \bw_h|_K \in \RT_1(K) \quad\forall K \in \cT_h \text{ and } 
    \bw_h \cdot \bn_\Omega = 0 \text{ on }\GammaEbz\cup\GammaNzbz \},
\\    
\widetilde Q_\bz &= \{ q_h \in L^2(\omega_\bz\cap\Omega_0) : q_h |_K \in P_1(K) \quad\forall K \in \cT_\bz, K \subset\overline\Omega_0\}.
\end{align*}
We seek $\bsigma_{h,i} \in \tbWh$
in the form \eqref{eq:sigsum}, where $(\bsigma_{\bz,i},q_{\bz,i}) \in \tbWz \times \widetilde Q_\bz$ are such that
\begin{multline}
\label{eq:sigB1loc}
  \left(\cA^{-1} \bsigma_{\bz,i}, \bw_h\right)_{\omega_\bz} 
  + \left( \frac{\ddiv\bsigma_{\bz,i}}{c+\gamma\beta_1}, \ddiv \bw_h \right)_{\omega_\bz\cap\Omega_+}
  + \left( \frac{\bsigma_{\bz,i}\cdot\bn_\Omega}{\alpha+\gamma\beta_2} , \bw_h\cdot\bn_\Omega \right)_{\GammaNpbz}
\\
  + (q_{\bz,i}, \ddiv\bw_h)_{\omega_\bz\cap\Omega_0}
  = 
  \left( \psi_\bz\frac{\nabla u_{h,i}}{\Lambda_{h,i}+\gamma}, \bw_h \right)_{\omega_\bz}
  - \left( \frac{(\Lambda_{h,i}\beta_1 - c)\psi_\bz u_{h,i}}{(c+\gamma\beta_1)(\Lambda_{h,i}+\gamma)} , \ddiv\bw_h \right)_{\omega_\bz\cap\Omega_+}
\\  
+ \left( \frac{(\cA \nabla \psi_\bz) \cdot \nabla u_{h,i}}{(c+\gamma\beta_1)(\Lambda_{h,i}+\gamma)}, \ddiv\bw_h\right)_{\omega_\bz\cap\Omega+}  
  + \left( \frac{(\Lambda_{h,i}\beta_2 - \alpha)\psi_\bz u_{h,i}}{(\alpha+\gamma\beta_2)(\Lambda_{h,i}+\gamma)} ,  \bw_h\cdot\bn_\Omega \right)_{\GammaNpbz}
\end{multline}
for all $\bw_h \in \tbWz$ and
\begin{equation}
\label{eq:sigB2loc}
  (\ddiv \bsigma_{\bz,i}, \varphi_h )_{\omega_\bz\cap\Omega_0}
%    = 0
 = \left( \frac{(\cA \nabla \psi_\bz) \cdot \nabla u_{h,i}}{\Lambda_{h,i}+\gamma}, \varphi_h\right)_{\omega_\bz\cap\Omega_0}      
%      - ( \beta_1 u_{h,i}, \varphi_h)_{\Omega_0} 
  \quad \forall \varphi_h \in \widetilde Q_\bz.  
\end{equation}

\begin{lemma}
\label{le:sigBlocOK}
Let $\bsigma_{\bz,i} \in \tbWz$ be solutions of problems \eqref{eq:sigB1loc}--\eqref{eq:sigB2loc} for all $\bz\in\cN_h$ and let $\bsigma_{h,i}$ be given by \eqref{eq:sigsum}.
Then $\bsigma_{h,i} \in \tbWh$ and it satisfies all assumptions of Theorem~\ref{th:BMPW}.
\end{lemma}
\begin{proof}
Zero normal components on edges $E\subset\GammaEbz$ enable to extend $\bsigma_{\bz,i} \in \bWz$ by zero such that the extension lies in $\Hdiv$ and consequently $\bsigma_{h,i}$ given by \eqref{eq:sigsum} lies in $\Hdiv$ as well.

The first condition in \eqref{eq:sigmacond} follows form \eqref{eq:sigB2loc}, the fact that $\ddiv \bsigma_{h,i}|_{\omega_\bz\cap\Omega_0}$ lies in $\widetilde Q_\bz$ and that piecewise constant $\beta_1=0$ in $\Omega_0$:
$$
    \norm{\ddiv \bsigma_{h,i} }_{L^2(\Omega_0)}^2 
  = \sum_{\bz\in\cN_h} (\ddiv \bsigma_{\bz,i}, \ddiv \bsigma_{h,i})_{\omega_\bz \cap \Omega_0}
  = \sum_{\bz\in\cN_h} \left( \frac{(\cA \nabla \psi_\bz) \cdot \nabla u_{h,i}}{\Lambda_{h,i}+\gamma}, \ddiv \bsigma_{h,i}\right)_{\omega_\bz\cap\Omega_0}      
  = 0.
$$
The second condition in \eqref{eq:sigmacond} is immediate form the requirements on normal components on $\GammaNzbz$ in the definition of $\tbWz$.
\end{proof}

%In this case the saddle point problem \eqref{eq:sigB1}--\eqref{eq:sigB2} reduces to
Local saddle point problems \eqref{eq:sigB1loc}--\eqref{eq:sigB2loc} simplify to the following positive definite problems provided conditions
\eqref{eq:cbeta1pos} are satisfied: find $\bsigma_{h,i} \in \tbWh$
in the form \eqref{eq:sigsum}, where $\bsigma_{\bz,i} \in \tbWz$ are such that
\begin{multline}
\label{eq:sigCloc}
  \left(\cA^{-1} \bsigma_{\bz,i}, \bw_h\right)_{\omega_\bz}
  + \left( \frac{\ddiv\bsigma_{\bz,i}}{c+\gamma\beta_1} , \ddiv \bw_h \right)_{\omega_\bz}
  + \left( \frac{\bsigma_{\bz,i}\cdot\bn_\Omega}{\alpha+\gamma\beta_2}, \bw_h\cdot\bn_\Omega \right)_{\GammaNpbz}
\\  
  = 
  \left( \psi_\bz\frac{\nabla u_{h,i}}{\Lambda_{h,i}+\gamma}, \bw_h \right)_{\omega_\bz}
  - \left( \frac{(\Lambda_{h,i}\beta_1 - c)\psi_\bz u_{h,i}}{(c+\gamma\beta_1)(\Lambda_{h,i}+\gamma)} , \ddiv \bw_h \right)_{\omega_\bz}
\\  
  + \left( \frac{(\cA \nabla \psi_\bz) \cdot \nabla u_{h,i}}{(c+\gamma\beta_1)(\Lambda_{h,i}+\gamma)}, \ddiv\bw_h\right)_{\omega_\bz}  
  + \left( \frac{(\Lambda_{h,i}\beta_2 - \alpha)\psi_\bz u_{h,i}}{(\alpha+\gamma\beta_2)(\Lambda_{h,i}+\gamma)} ,  \bw_h\cdot\bn_\Omega \right)_{\GammaNpbz}
\end{multline}
for all $\bw_h \in \tbWz$.
The fact that this flux reconstruction satisfies all assumptions of Theorem~\ref{th:BMPW} follows from Lemma~\ref{le:sigBlocOK} as a special case.
%We note that it is easy to verify that local flux reconstructions based on both \eqref{eq:sigB1loc}--\eqref{eq:sigB2loc} and \eqref{eq:sigCloc} satisfy conditions \eqref{eq:sigmacond}, cf. Lemma~\ref{le:locsigcons}.

We now summarize the Lehmann--Goerisch method for the general eigenvalue problem as an algorithm for computing lower bounds $\ell_i$, $i=1,2,\dots,m$, on the first $m$ eigenvalues. 

\bigskip

\noindent{\bf Algorithm 1.}
\begin{enumerate}
\item 
Let $\ell_{m+1} \leq \lambda_{m+1}$ be an \emph{a priori} known lower bound and let $\gamma >0$ be a fixed parameter. 
\item 
Compute standard finite element approximations \eqref{eq:fem} of the first $m$ eigenpairs $(\Lambda_{h,i}, u_{h,i}) \in \R \times V_h$, $i=1,2,\dots,m$. This provides upper bounds $\Lambda_{h,i}$, $i=1,2,\dots,m$, on the exact eigenvalues.
\item
Find $\bsigma_{h,i} \in \bWh$ (or $\tbWh$) for $i=1,2,\dots,m$ by solving one of the following problems:
\begin{itemize}
\item[(a)] global saddle point problem \eqref{eq:sig1}--\eqref{eq:sig2},
\item[(b)] global saddle point problem \eqref{eq:sigB1}--\eqref{eq:sigB2}, 
\item[(c)] global positive definite problem \eqref{eq:sigC}, provided condition \eqref{eq:cbeta1pos} is satisfied,
\item[(d)] local saddle point problems \eqref{eq:sig1loc}--\eqref{eq:sig2loc} and using \eqref{eq:sigsum},
\item[(e)] local saddle point problems \eqref{eq:sigB1loc}--\eqref{eq:sigB2loc} and using \eqref{eq:sigsum},
\item[(f)] local positive definite problems \eqref{eq:sigCloc} and using \eqref{eq:sigsum}, provided condition \eqref{eq:cbeta1pos} is satisfied.
\end{itemize}
%\item
%For all $n = m, m-1,\dots,1$ do the following steps.
\item Set $\rho = \ell_{m+1} + \gamma$.
\item Assemble matrices $\bM, \bN \in \R^{m\times m}$ using $\tilde u_i = u_{h,i}$ and $\tbsigma_i = \bsigma_{h,i}$ for $i=1,2,\dots,m$ as in Theorem~\ref{th:BMPW}.
\item Find eigenvalues $\mu_1 \leq \mu_2 \leq \cdots \leq \mu_m$ of \eqref{eq:MNproblem}.
\item 
If $\bN$ is not positive definite then set $\ell_j=-\infty$ for all $j = 1,2,\dots,m$.
\\
Otherwise
use \eqref{eq:llowinc} with $L=m+1$, $i=m+1-j$, $j=1,2,\dots,m$, and compute
$$
%   \ell^*_{j,n} = \rho^{(n)} - \gamma^{(n)} - \rho^{(n)}/\left(1-\mu^{(n)}_{n+1-j}\right) \leq \lambda_{j}
    \ell_j = \left\{\begin{array}{ll}
      \rho - \gamma - \rho/\left(1-\mu_{m+1-j}\right) & \text{if } \mu_{m+1-j} < 0, \\
      -\infty & \text{otherwise.}
    \end{array}\right.
$$
%Note that $\ell_j \leq \lambda_j$ for all $j = 1,2,\dots,m$.
% \item 
%\item Find the best available lower bound $\ell_n = \max \{ \ell^*_{n,i}, \ i = n, n+1,\dots,m \} \leq \lambda_n$.
\end{enumerate}

The output of this algorithm consists of two-sided bounds on the first $m$ eigenvalues:
$$
  \ell_i \leq \lambda_i \leq \Lambda_{h,i}, \quad i=1,2,\dots,m.
$$
The relative eigenvalue enclosure size
\begin{equation}
\label{eq:relencl}
  (\Lambda_{h,i} - \ell_i)/\ell_i
\end{equation}  
bounds the true relative error and it
is used below in Sections~\ref{se:numex}--\ref{se:Steklov} as a measure of the accuracy of the method.
Let us note that if the \emph{a~priori} lower bound $\ell_{m+1}$ on $\lambda_{m+1}$ is too rough, typically if $\ell_{m+1} \leq \lambda_m$ then 
it may happen that Algorithm~1 still computes a positive lower bound $\ell_i$ on $\lambda_i$ for some $i$, but it will often be rough and will not converge to $\lambda_i$, but to a smaller eigenvalue.
Alternatively, it may happen that the assumptions on the positive definiteness of $\bN$ and/or on the negativity of $\mu_i$ are not satisfied and the algorithm returns $\ell_i = -\infty$ for some $i$.

Lemmas~\ref{le:sigOK}, \ref{le:sigBOK}, \ref{le:locsigcons}, and \ref{le:sigBlocOK} verify that all flux reconstructions presented in step 3 of Algorithm~1 satisfy assumptions of Theorem~\ref{th:BMPW}, which justifies that this algorithm produces lower bounds on eigenvalues.
In this paper we assume that matrices $\bM$ and $\bN$ in step 5, eigenvalues $\mu_1, \dots, \mu_m$ in step 6, and lower bounds $\ell_j$ in step 7 are computed exactly. If these computations are performed in the floating point arithmetic then they are polluted by round off errors and the computed lower bounds need not be guaranteed to be below the true eigenvalues. This problem can be solved by employing interval arithmetic as proposed for example in \cite{Plum1990,Plum1991,Liu2015}. We just note that the interval arithmetic is only needed in steps 4--7 of the Algorithm~1, where the most involved part is the solution of the small generalized eigenvalue problem with matrices $\bM$ and $\bN$. The finite element approximations $u_{h,i}$ in step 2 and flux reconstructions $\bsigma_{h,i}$ in step 3 can be polluted by various errors, because Theorem~\ref{th:BMPW} allows for arbitrary $\tilde u_i$ and $\tbsigma_i$.

\section{Numerical example -- Laplace eigenvalue problem in the dumbbell shaped domain}
\label{se:numex}

In this section, we compare the accuracy and computational performance of global and local flux reconstructions presented above. As an example we choose two-dimensional Laplace eigenvalue problem \eqref{eq:Laplace} in a dumbbell shaped domain \cite{TreBet2006}.
% In particular, we choose $\cA=I$, $c=0$, $\beta_1=1$, $\GammaN=\emptyset$, and $\GammaD=\partial\Omega$ in \eqref{eq:EPstrong}. 
This domain can be expressed as
$\Omega = (0,\pi)^2 \cup \left( [\pi,5\pi/4]\times(3\pi/8,5\pi/8) \right) \cup \left( (5\pi/4,9\pi/4)\times(0,\pi) \right)$
and it is illustrated in Figure~\ref{fi:dmbl} (left).

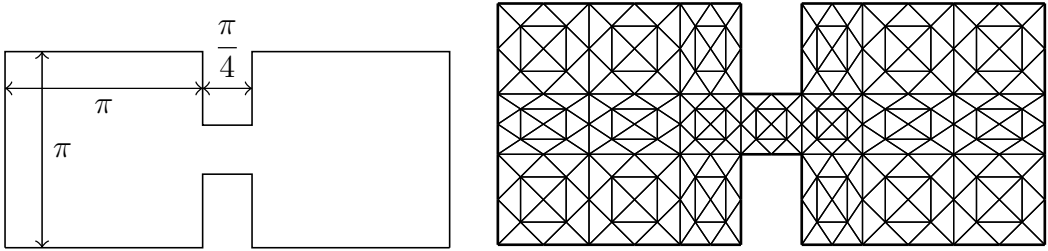
\begin{figure}
\tikzset{
  big arrowhead/.style={
    decoration={markings, mark=at position 1 with {\arrow[scale=1.5,thin,black]{>}} },
    postaction={decorate},
    shorten >=0.4pt}}%
\begin{tikzpicture}[scale=0.65]
\draw [semithick] (0,0)--(4,0)--(4,1.5)--(5,1.5)--(5,0)--(9,0);
\draw [semithick] (9,0)--(9,4)--(5,4)--(5,2.5)--(4,2.5)--(4,4)--(0,4)--(0,0);
%\draw [<-{>[scale=2]}, thin] (0,0.5)--(3,0.5);
\draw [big arrowhead, thin] (0,3.25)--(4,3.25);
\draw [big arrowhead, thin] (4,3.25)--(0,3.25);
\node [below] at (2,3.25) {$\pi$};
\draw [big arrowhead, thin] (4,3.25)--(5,3.25);
\draw [big arrowhead, thin] (5,3.25)--(4,3.25);
\node [above] at (4.5,3.25) {$\displaystyle\frac{\pi}{4}$};
%\draw [big arrowhead, thin] (4,2.5)--(7,2.5);
%\draw [big arrowhead, thin] (7,2.5)--(4,2.5);
%\node [below] at (5.5,2.5) {$\pi$};
\draw [big arrowhead, thin] (0.75,0)--(0.75,4);
\draw [big arrowhead, thin] (0.75,4)--(0.75,0);
\node [right] at (0.75,2) {$\pi$};
%\draw [big arrowhead, thin] (3.5,1)--(3.5,2);
%\draw [big arrowhead, thin] (3.5,2)--(3.5,1);
%\node [right] at (3.5,1.5) {$\displaystyle\frac{\pi}{3}$};
\end{tikzpicture}
\quad
\includegraphics[width=0.49\textwidth]{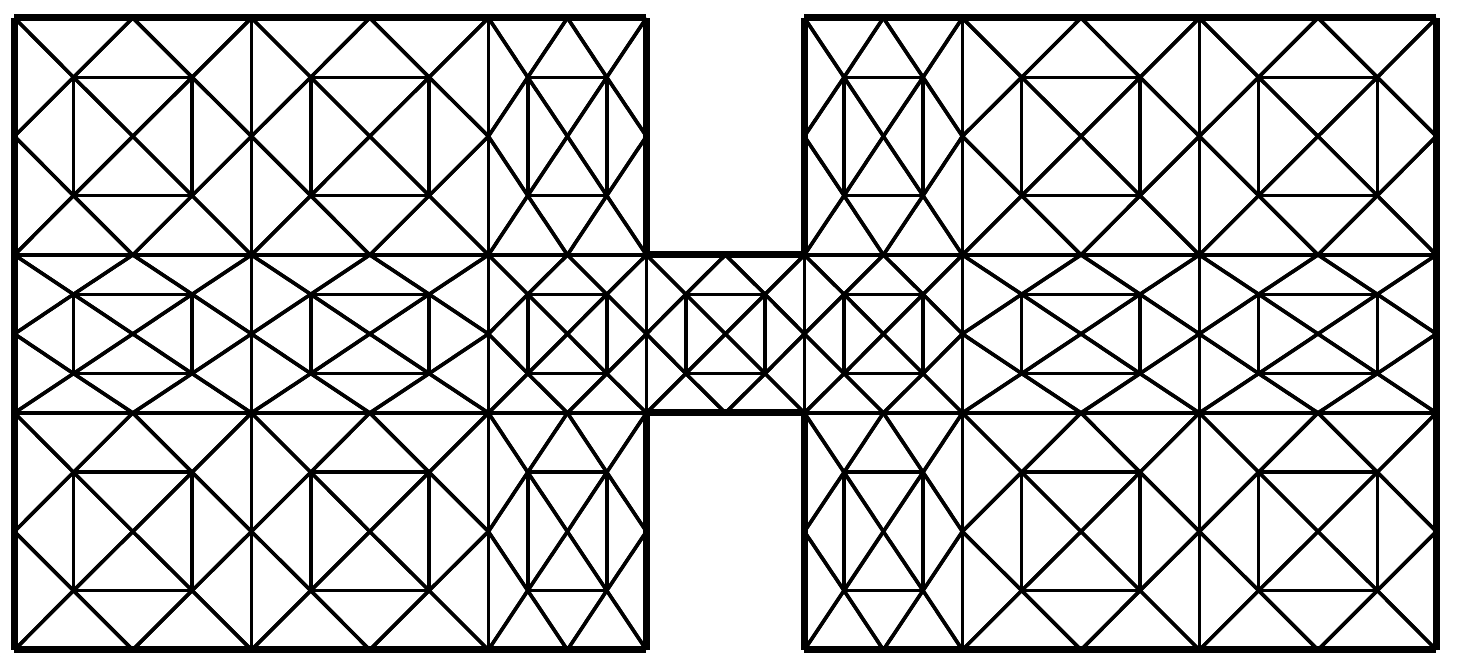}%
\caption{\label{fi:dmbl}
The dumbbell shaped domain $\Omega$ (left) and its initial triangulation (right).
}
\end{figure}

We compute the first $m=6$ eigenvalues of this problem by the standard finite element method \eqref{eq:Lapfem} and the corresponding lower bounds by the Lehmann--Goerisch method with four flux reconstructions presented in Sections~\ref{se:LGLap}--\ref{se:simpleLap}. We use Algorithm~1 described at the end of Section~\ref{se:local}.
We perform these computations on a series of uniformly refined meshes starting with the mesh depicted in Figure~\ref{fi:dmbl} (right).
The shift parameter $\gamma$ is recommended to be small \cite{BehMerPluWie2000} and we choose $\gamma = 10^{-6}$.

The \emph{a priori} known lower bound on the exact eigenvalue $\lambda_{m+1}$ is computed by using the monotonicity principle. 
We enclose the dumbbell shaped domain $\Omega$ into a rectangle $\mathcal{R} = (0,9\pi/4)\times(0,\pi)$.
The Laplace eigenvalue problem in $\mathcal{R}$ can be solved analytically and because $\Omega \subset \mathcal{R}$, the eigenvalues of the Laplacian on $\mathcal{R}$ lie below the corresponding eigenvalues on $\Omega$.
This simple approach is sufficient for the first six eigenvalues, because the seventh eigenvalue on the rectangle $\lambda_7^{(\mathcal{R})} \approx 5.778$ is still above the sixth eigenvalue for $\Omega$. This is no longer the case for higher eigenvalues, which can be verified by computing sufficiently accurate upper bounds $\Lambda_{h,i}$ by \eqref{eq:Lapfem} for the dumbbell shaped domain $\Omega$.

% In particular, we use the eleventh eigenvalue $\ell_{11} = \lambda_{11}^{(\mathcal{R})} \approx 8.938$ on the rectangle $\mathcal{R}$ as the \emph{a priori} known lower bound for $\lambda_{11}$ on the dumbbell shaped domain $\Omega$.
% Let us note that for the dumbbell shaped domain we have $\lambda_8 \approx 7.99$, $\lambda_9 \approx 9.36$, and $\lambda_{10} \approx 9.51$. Thus, $\ell_{11} \approx 8.938$ is a very rough lower bound on $\lambda_{11}$. It lies even below $\lambda_9$ and $\lambda_{10}$, and therefore the Lehmann--Goerisch method cannot provide good lower bounds on $\lambda_9$ and $\lambda_{10}$. However, we can still use it to compute some lower bounds on $\lambda_3$, \dots, $\lambda_10$.

Numerical results below compare the global flux reconstruction \eqref{eq:sigLap1}--\eqref{eq:sigLap2} and the local flux reconstruction \eqref{eq:sigLap1loc}--\eqref{eq:sigLap2loc} with their simplified and positive definite versions \eqref{eq:sigLapB} and \eqref{eq:sigLapBloc}. Notice that we can use these simplified versions, because the Laplace eigenvalue problem satisfies condition \eqref{eq:cbeta1pos}.

% which have both the saddle point structure and are solved by the mixed finite elements
% 
% Condition \eqref{eq:cbeta1pos} is satisfied for the Laplace eigenvalue problem, therefore flux reconstructions 
% \eqref{eq:sigB1}--\eqref{eq:sigB2} and \eqref{eq:sigB1loc}--\eqref{eq:sigB2loc} simplify to 
% \eqref{eq:sigC} and \eqref{eq:sigCloc}, respectively.

Figure~\ref{fi:dmbl_encl1} shows the relative enclosure size \eqref{eq:relencl} for $\lambda_1$,
where the lower bound $\ell_1$ is
computed by using these four flux reconstructions. 
% obtained by solving the
% global saddle point problem \eqref{eq:sig1}--\eqref{eq:sig2}, 
% global positive definite problem \eqref{eq:sigC}, 
% local saddle point problem \eqref{eq:sig1loc}--\eqref{eq:sig2loc}, 
% and local positive definite problem \eqref{eq:sigCloc}.
The left panel presents the dependence of these enclosure sizes on the mesh size $h = \max_{K\in\cT_h} \operatorname{diam} K$.
We observe that all four flux reconstructions provide virtually the same results on a given mesh. 
However, the computational performance of these approaches considerably differs. Especially the memory requirements of global flux reconstructions \eqref{eq:sig1}--\eqref{eq:sig2} and \eqref{eq:sigC} are substantially larger than the memory requirements of local flux reconstructions \eqref{eq:sig1loc}--\eqref{eq:sig2loc} and \eqref{eq:sigCloc}.
Therefore, we present in the right panel of Figure~\ref{fi:dmbl_encl1} the dependence of the same relative enclosure sizes on the number of degrees freedom. Specifically, the number of degrees of freedom for the global saddle point problem \eqref{eq:sig1}--\eqref{eq:sig2} is the dimension of $\bWh^0$ plus the dimension of $Q_h$. For the global positive definite problem \eqref{eq:sigC} it is the dimension of $\tbWh$ only, and for both local flux reconstructions it is the dimension of $V_h$.

\begin{figure}
\includegraphics[height=0.385\textwidth]{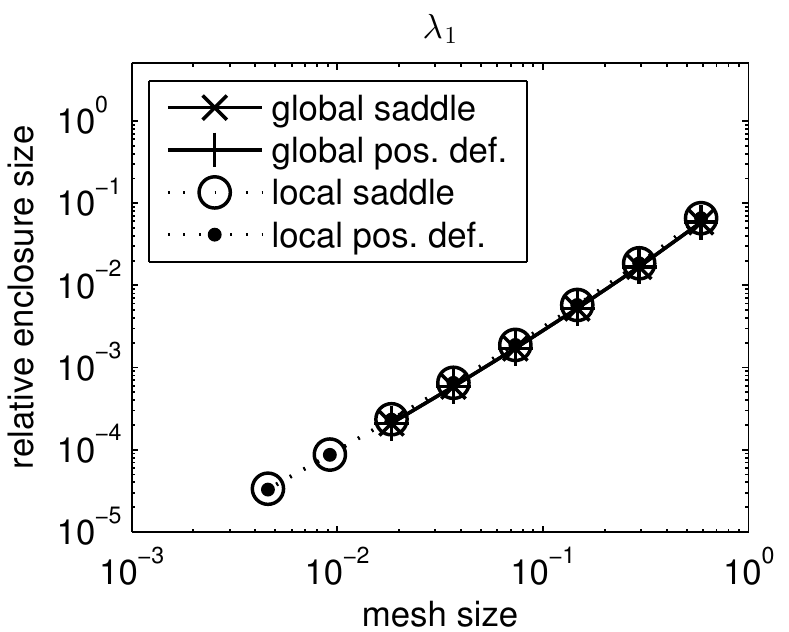}\quad%
\includegraphics[height=0.385\textwidth]{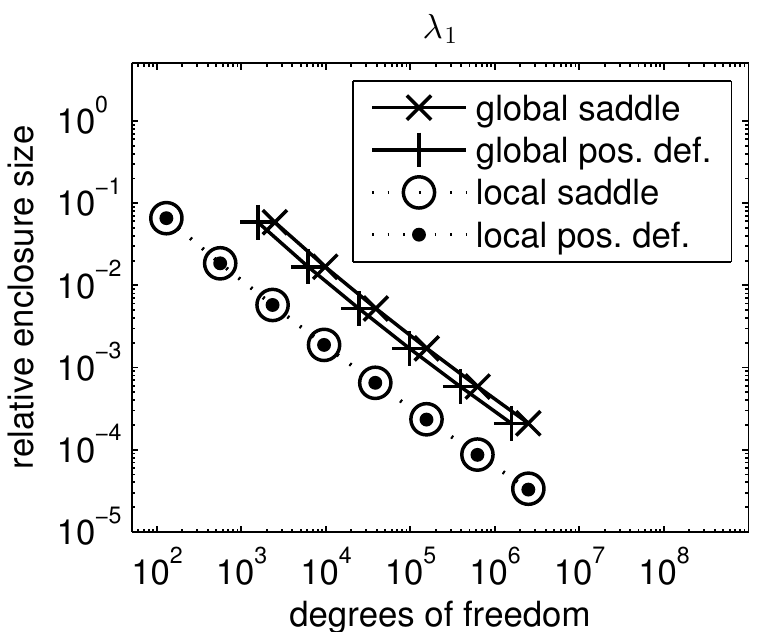}%
\caption{\label{fi:dmbl_encl1}
Dependence of the relative enclosure size $(\Lambda_{h,1} - \ell_1)/\ell_1$ on the mesh size (left) and on the number of degrees of freedom (right) for the first eigenvalue of the Laplacian on the dumbbell shaped domain. 
The four curves correspond to flux reconstructions computed by solving 
the global saddle point problem \eqref{eq:sig1}--\eqref{eq:sig2}, 
global positive definite problem \eqref{eq:sigC}, 
local saddle point problem \eqref{eq:sig1loc}--\eqref{eq:sig2loc}, 
and local positive definite problem \eqref{eq:sigCloc}.
}
\end{figure}

Concerning the higher eigenvalues,
the four flux reconstructions yield almost the same results as in the case of the first eigenvalue. For illustration we present the relative enclosure size \eqref{eq:relencl} for the fifth eigenvalue in Figure~\ref{fi:dmbl_encl5}. We emphasize that the spectral gap between $\lambda_5$ and $\lambda_6$ is extremely small for the dumbbell shaped domain and therefore the lower bound on $\lambda_5$ is less accurate than lower bounds on the other eigenvalues. In any case, the four tested flux reconstructions 
are almost identically accurate, see Figure~\ref{fi:dmbl_encl5} (left), 
and the corresponding dependence on the number of degrees of freedom 
in Figure~\ref{fi:dmbl_encl5} (right) reflects the memory requirements.

\begin{figure}
\includegraphics[height=0.385\textwidth]{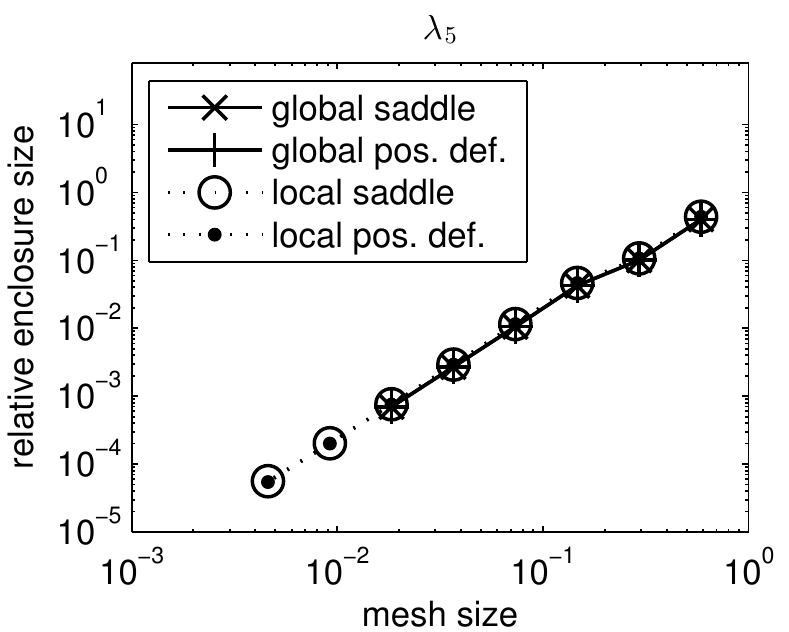}\quad%
\includegraphics[height=0.385\textwidth]{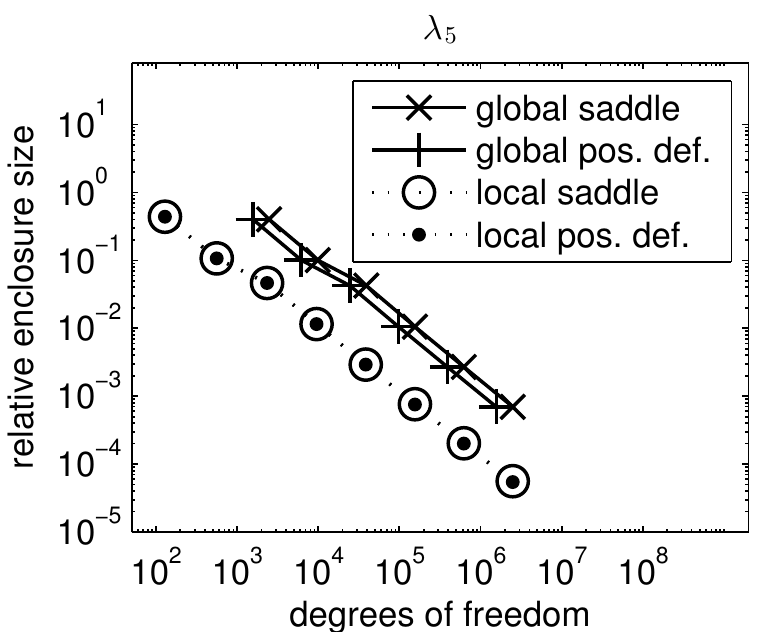}%
\caption{\label{fi:dmbl_encl5}
Dependence of the relative enclosure size $(\Lambda_{h,5} - \ell_5)/\ell_5$ on the mesh size (left) and on the number of degrees of freedom (right) for the fifth eigenvalue of the Laplacian on the dumbbell shaped domain. 
The four curves correspond to flux reconstructions computed by solving 
the global saddle point problem \eqref{eq:sig1}--\eqref{eq:sig2}, 
global positive definite problem \eqref{eq:sigC}, 
local saddle point problem \eqref{eq:sig1loc}--\eqref{eq:sig2loc}, 
and local positive definite problem \eqref{eq:sigCloc}.
}
\end{figure}

Notice that on the two finest meshes we could not solve global flux reconstruction problems, because of the lack of computer memory. In contrast, the local problems need virtually no additional memory and we can solve them even on the finest meshes.
The left panels of Figures~\ref{fi:dmbl_encl1} and \ref{fi:dmbl_encl5}
confirm that the solution of local problems does not compromise the accuracy of the resulting lower bounds.

The accuracy of the four flux reconstructions is compared in Table~\ref{ta:dumbbell}, where the corresponding lower bounds together with the finite element upper bound are listed. The presented results are computed on the six times refined uniform mesh, which was the finest mesh, where we were able to compute all four flux reconstructions. This table confirms that all flux reconstructions provide similar accuracy. The local reconstructions yield naturally less accurate lower bounds then the global reconstructions, but the differences between the lower bounds computed by local and global reconstructions represents only around 10\,\% of the resulting eigenvalue enclosures.
Nevertheless, the main advantage of local reconstructions is that they enable to refine the mesh two times more and the gain in accuracy is visible in Figures~\ref{fi:dmbl_encl1} and \ref{fi:dmbl_encl5}.
\begin{table}
\begin{tabular}{cccccc}
 & glob. saddle & glob. pos. def. & loc. saddle & loc. pos. def. & FEM \\
\hline
$\lambda_1$ & 1.955616813 & 1.955619836 & 1.955569884 & 1.955572909 & 1.956027811 \\
$\lambda_2$ & 1.960523818 & 1.960526445 & 1.960482057 & 1.960485085 & 1.960894364 \\
$\lambda_3$ & 4.793800128 & 4.793811934 & 4.792874441 & 4.792886284 & 4.801978452 \\
$\lambda_4$ & 4.823503783 & 4.823515952 & 4.822671400 & 4.822683594 & 4.830982305 \\
$\lambda_5$ & 4.993812020 & 4.993826800 & 4.993513785 & 4.993528575 & 4.997300028 \\
$\lambda_6$ & 4.993826895 & 4.993841675 & 4.993528825 & 4.993543614 & 4.997313686 \\
\end{tabular}
\caption{\label{ta:dumbbell}
Lower bounds for the Laplace eigenvalue problem in the dumbbell shaped domain computed by global saddle point problem \eqref{eq:sigLap1}--\eqref{eq:sigLap2}, global positive definite problem \eqref{eq:sigLapB}, local saddle point problem \eqref{eq:sigLap1loc}--\eqref{eq:sigLap2loc}, and local positive definite problem \eqref{eq:sigLapBloc}. The last column presents the upper bound computed by the finite element method \eqref{eq:Lapfem}.
}
\end{table}

\section{Numerical example -- Steklov-type eigenvalue problem}
\label{se:Steklov}

This section illustrates the accuracy and numerical performance of the presented flux reconstructions for a Steklov-type eigenvalue problem. We again consider the dumbbell shaped domain $\Omega$, but this time with mixed Dirichlet and Neumann boundary conditions. We consider the left-most edge of $\partial\Omega$ to be the Neumann part of the boundary $\GammaN = \{0\} \times (0,\pi)$ and the rest of the boundary to be the Dirichlet part $\GammaD = \partial\Omega \setminus \GammaN$. The Steklov-type eigenvalue problem we will solve is a special case of \eqref{eq:EPstrong} with parameters $\cA = I$, $c=0$, $\beta_1=0$ in $\Omega$ and $\alpha = 0$, $\beta_2 = 1$ on $\GammaN$. The shift parameter is chosen again as $\gamma = 10^{-6}$.

The \emph{a~priori} known lower bound can be computed by the monotonicity principle and by enclosing $\Omega$ into the same rectangle $\mathcal{R}$ as in Section~\ref{se:numex}. The Steklov-type eigenvalue problem in the rectangle $\mathcal{R}$ (with $\GammaN$ representing the Neumann part of the boundary) can be solve analytically and we have $\lambda_k^\mathcal{R} = k \coth(9k\pi/4)$, $k=1,2,\dots$. Choosing the seventh eigenvalue on the rectangle $\lambda_7^\mathcal{R} \approx 7.000$ as a guaranteed lower bound on $\lambda_7$ on the dumbbell shaped domain, we compute lower bounds on the first six eigenvalues by employing Algorithm~1.

Notice that in this setting we have $\Omega_0 = \Omega$, $\Omega_+ = \emptyset$, and condition \eqref{eq:cbeta1pos} is not satisfied. Therefore, the positive definite variants of flux reconstructions are not available and
we use flux reconstructions obtained by solving global saddle point problems \eqref{eq:sig1}--\eqref{eq:sig2}, \eqref{eq:sigB1}--\eqref{eq:sigB2}, and local saddle point problems \eqref{eq:sig1loc}--\eqref{eq:sig2loc}, \eqref{eq:sigB1loc}--\eqref{eq:sigB2loc}.

Since $c=\beta_1=0$ and $\Omega_0 = \Omega$, equations \eqref{eq:sig2} and \eqref{eq:sigB2} are identical. Thus, the only difference between the two global saddle point problems is in the handling of normal components of fluxes on $\GammaN$. Problem \eqref{eq:sig1}--\eqref{eq:sig2} considers them as essential boundary conditions incorporated in the definition of the space $\bWh$, while problem \eqref{eq:sigB1}--\eqref{eq:sigB2} enforces their correct values by the penalty method. The difference between the two local flux reconstructions is of the same nature.

Figure~\ref{fi:Stek_encl} presents the corresponding convergence curves for $\lambda_1$ and $\lambda_5$ with respect to both the mesh size and the number of degrees of freedom. 
As in the case of the Laplace eigenvalue problem, all flux reconstructions provide almost the same accuracy on a fixed mesh, see left panes of Figure~\ref{fi:Stek_encl}. However, global problems require considerably more degrees of freedom, see right panels of Figure~\ref{fi:Stek_encl}, and we are not able to solve them on the two finest meshes.

\begin{figure}
\includegraphics[height=0.385\textwidth]{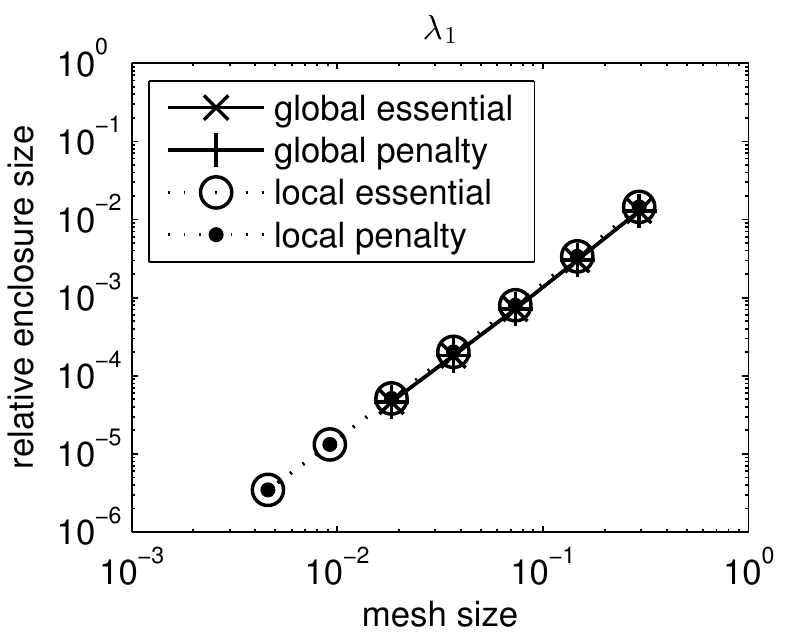}\quad%
\includegraphics[height=0.385\textwidth]{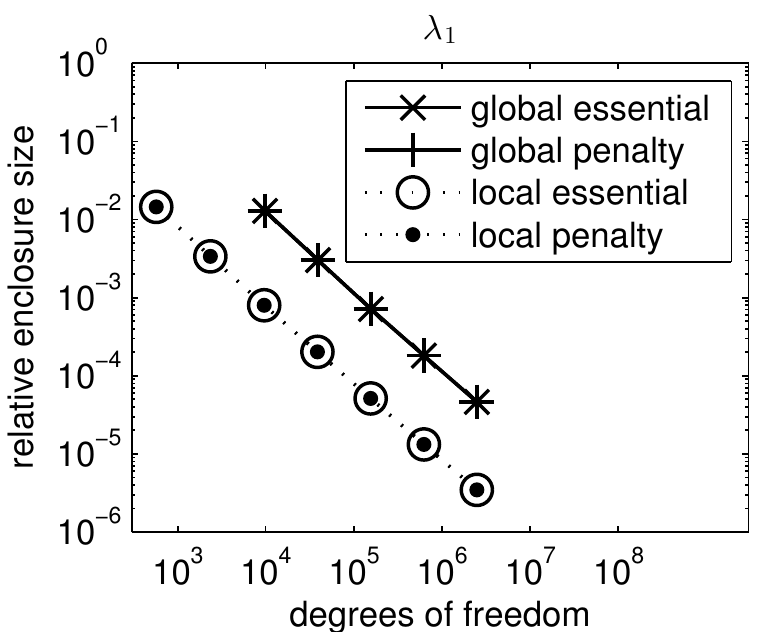}%
\\
\includegraphics[height=0.385\textwidth]{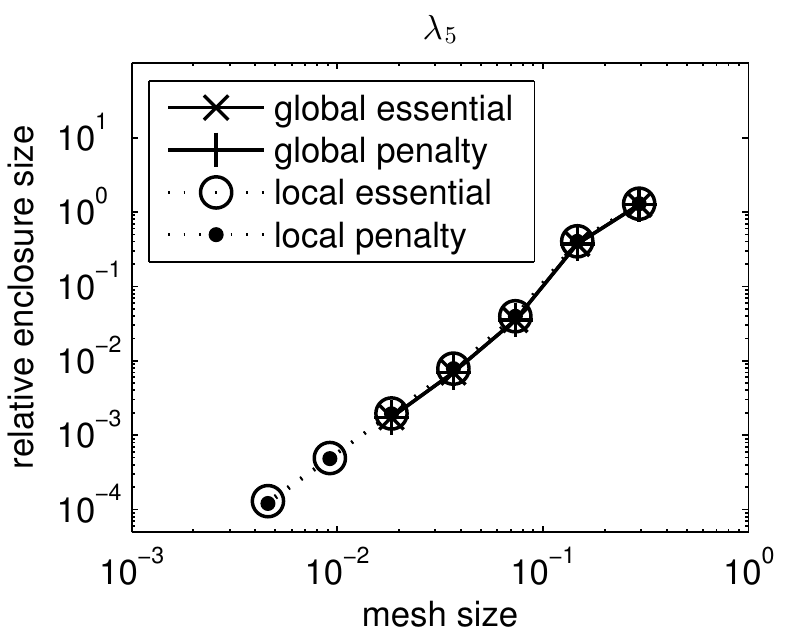}\quad%
\includegraphics[height=0.385\textwidth]{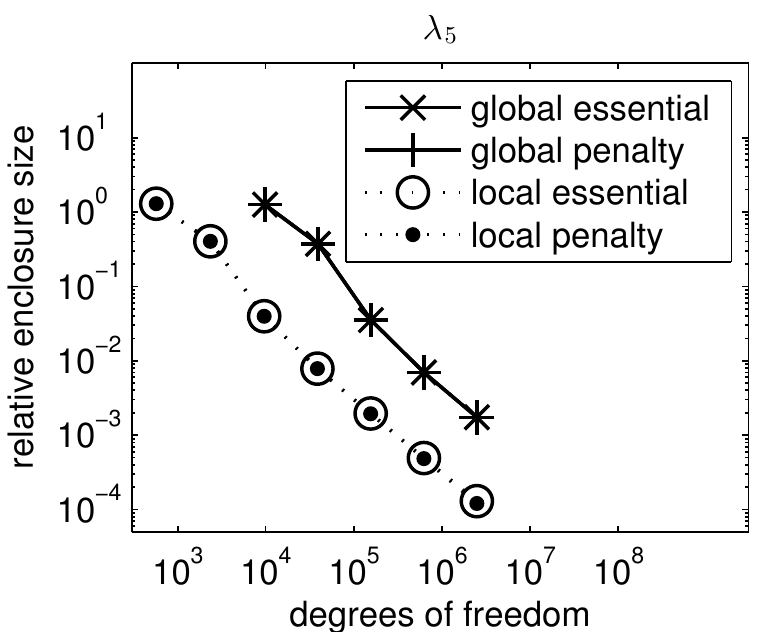}%
\caption{\label{fi:Stek_encl}
Dependence of the relative enclosure sizes $(\Lambda_{h,1} - \ell_1)/\ell_1$ (top row) and $(\Lambda_{h,5} - \ell_5)/\ell_5$ (bottom row) on the mesh size (left) and on the number of degrees of freedom (right) for the first and the fifth eigenvalue of the Steklov type eigenvalue problem on the dumbbell shaped domain. 
The four curves correspond to flux reconstructions computed by solving 
the global saddle point problem \eqref{eq:sig1}--\eqref{eq:sig2} with essential boundary conditions on $\GammaN$, 
the global saddle point problem \eqref{eq:sigB1}--\eqref{eq:sigB2} with the penalty parameter, 
local saddle point problems \eqref{eq:sig1loc}--\eqref{eq:sig2loc} with essential boundary conditions on $\GammaN$, 
and local saddle point problems \eqref{eq:sigB1loc}--\eqref{eq:sigB2loc} with the penalty parameter. 
}
\end{figure}

Table~\ref{ta:Steklov} compares lower bounds obtained by the four flux reconstructions for the first six eigenvalues as they were computed on the six times refined initial mesh.
Global flux reconstructions provide slightly more accurate lower bounds, but the difference of the lower bounds obtained by global and local reconstructions is again around 10\,\% of the size of the eigenvalue enclosure.

\begin{table}
\begin{tabular}{cccccc}
 & glob. essen. & glob. penalty & loc. essen. & loc. penalty & FEM\\
\hline
$\lambda_1$ & 1.003284998 & 1.003284998 & 1.003279585 & 1.003279585 & 1.003334201 \\
$\lambda_2$ & 1.999883355 & 1.999883355 & 1.999827309 & 1.999831448 & 2.000339499 \\
$\lambda_3$ & 2.999234430 & 2.999234430 & 2.999019928 & 2.999033731 & 3.001020719 \\
$\lambda_4$ & 3.996605934 & 3.996605934 & 3.995891502 & 3.995925975 & 4.002545124 \\
$\lambda_5$ & 4.988104630 & 4.988104630 & 4.986113993 & 4.986196556 & 5.004758449 \\
$\lambda_6$ & 5.950671350 & 5.950671350 & 5.943809247 & 5.944048237 & 6.008222917 \\
\end{tabular}
\caption{\label{ta:Steklov}
Lower bounds for the Steklov type eigenvalue problem in the dumbbell shaped domain computed by the global flux reconstruction \eqref{eq:sig1}--\eqref{eq:sig2} with essential boundary conditions on $\GammaN$, global reconstruction \eqref{eq:sigB1}--\eqref{eq:sigB2} with the penalty parameter, local reconstruction \eqref{eq:sig1loc}--\eqref{eq:sig2loc} with essential boundary conditions on $\GammaN$, and local reconstruction \eqref{eq:sigB1loc}--\eqref{eq:sigB2loc} with the penalty parameter. The last column presents the upper bound computed by the finite element method \eqref{eq:fem}. 
}

\end{table}

% % Data for the table:
% % Refinement level == 6 (5th mesh, because I start with refinement step 2)
% 1. Global mixed problem
% 2. Global simplified (pos. def.) problem
% 3. Local mixed problems (Braess-Schoberl)
% 4. Local simplified (pos. def.) problems
% lambda_1: lambdaguarlow_1= 1.00000144989555 low1= 1.00328499810568 low2= 1.0032849981057  low3= 1.00327958544847 low4= 1.00327958544837 up= 1.00333420105546
% lambda_2: lambdaguarlow_2=  2.0000000000021 low1= 1.99988335508704 low2= 1.99988335508705 low3= 1.99982730902543 low4= 1.99983144854894 up= 2.00033949943327
% lambda_3: lambdaguarlow_3=                3 low1= 2.9992344306407  low2= 2.99923443064094 low3= 2.99901992862505 low4= 2.99903373152115 up= 3.00102071949967
% lambda_4: lambdaguarlow_4=                4 low1= 3.9966059343596  low2= 3.99660593436062 low3= 3.99589150205822 low4= 3.99592597597019 up= 4.00254512448091
% lambda_5: lambdaguarlow_5=                5 low1= 4.98810463020443 low2= 4.98810463020783 low3= 4.98611399376408 low4= 4.98619655627665 up= 5.00475844905588
% lambda_6: lambdaguarlow_6=                6 low1= 5.95067135011634 low2= 5.95067135012997 low3= 5.94380924752508 low4= 5.94404823789063 up= 6.00822291730227

\section{Conclusions}
\label{se:concl}

In this paper we propose alternative approaches for computing flux reconstructions in the Lehmann--Goerisch method. These alternative approaches are less computationally demanding and provide almost as accurate results as the traditional global approach. Flux reconstruction \eqref{eq:sigC} can be recommended for small problems, because it is simpler to implement and less computationally demanding than the traditional saddle point problem \eqref{eq:sig1}--\eqref{eq:sig2}. However, for large scale problems the local flux reconstructions are recommended, because the resulting local problems are independent and can be easily solved in parallel. Flux reconstruction \eqref{eq:sigCloc} is especially advantageous, because it requires to solve just a simple positive definite problem by standard Raviart--Thomas finite elements.

Let us mention that the presented approach is applicable to the general eigenvalue problem \eqref{eq:EPstrong} in arbitrary dimension, with variable coefficients, and mixed boundary conditions. For technical reasons connected with the specific flux reconstructions we assumed piecewise constant coefficients, however, the general idea is applicable even in the case of more general coefficients. Additional advantage of the presented approach is its suitability for generalizations to higher order approximations. Further, this approach can be well combined with mesh adaptivity and presented flux reconstructions can be used to compute local error indicators for mesh refinement.

From a wider perspective, this paper shows that the local and efficient flux reconstructions developed in the last decade for boundary value problems can be utilized in the Lehmann--Goerisch method in order to efficiently compute accurate lower bounds on eigenvalues. Current progress in constructing efficient flux reconstructions for more complex problems such as linear and nonlinear elasticity \cite{BerMolSta2017} promises their future utilization in corresponding eigenvalue problems for computing accurate lower bounds on eigenvalues.

% \section{The Elsevier article class}
% 
% \paragraph{Installation} If the document class \emph{elsarticle} is not available on your computer, you can download and install the system package \emph{texlive-publishers} (Linux) or install the \LaTeX\ package \emph{elsarticle} using the package manager of your \TeX\ installation, which is typically \TeX\ Live or Mik\TeX.
% 
% \paragraph{Usage} Once the package is properly installed, you can use the document class \emph{elsarticle} to create a manuscript. Please make sure that your manuscript follows the guidelines in the Guide for Authors of the relevant journal. It is not necessary to typeset your manuscript in exactly the same way as an article, unless you are submitting to a camera-ready copy (CRC) journal.
% 
% \paragraph{Functionality} The Elsevier article class is based on the standard article class and supports almost all of the functionality of that class. In addition, it features commands and options to format the
% \begin{itemize}
% \item document style
% \item baselineskip
% \item front matter
% \item keywords and MSC codes
% \item theorems, definitions and proofs
% \item lables of enumerations
% \item citation style and labeling.
% \end{itemize}
% 
% \section{Front matter}
% 
% The author names and affiliations could be formatted in two ways:
% \begin{enumerate}[(1)]
% \item Group the authors per affiliation.
% \item Use footnotes to indicate the affiliations.
% \end{enumerate}
% See the front matter of this document for examples. You are recommended to conform your choice to the journal you are submitting to.

\section*{Acknowledgement}
The author would like to thank anonymous referees for constructive comments that yield to substantial improvements of the paper.
Further the author gratefully acknowledges the support of Neuron Fund for Support of Science, project no.~24/2016
and the institutional support RVO~67985840.

% \section{Bibliography styles}
% 
% There are various bibliography styles available. You can select the style of your choice in the preamble of this document. These styles are Elsevier styles based on standard styles like Harvard and Vancouver. Please use Bib\TeX\ to generate your bibliography and include DOIs whenever available.
% 
% Here are two sample references: \cite{Feynman1963118,Dirac1953888}.

%\section*{References}

\bibliographystyle{amsplain} % bibliographic style - name of *.bst file to use
\bibliography{bibl}

\ifx\undefined\bysame
\newcommand{\bysame}{\leavevmode\hbox to3em{\hrulefill}\,}
\fi
\begin{thebibliography}{10}

\bibitem{BabOsb:1991}
Ivo Babu{\v{s}}ka and John~E. Osborn, {\em Eigenvalue problems}, Handbook of
  numerical analysis, {V}ol.\ {II} (Amsterdam), North-Holland, Amsterdam, 1991,
  pp.~641--787.

\bibitem{Barrenechea2014}
G.~R. Barrenechea, L.~Boulton, and N.~Boussa{\"{\i}}d, {\em Finite element
  eigenvalue enclosures for the {M}axwell operator}, SIAM J. Sci. Comput. {\bf
  36} (2014), no.~6, A2887--A2906.

\bibitem{BehnkeGoerish1994}
H.~Behnke and F.~Goerisch, {\em Inclusions for eigenvalues of selfadjoint
  problems}, Topics in validated computations ({O}ldenburg, 1993), Stud.
  Comput. Math., vol.~5, North-Holland, Amsterdam, 1994, pp.~277--322.

\bibitem{BehMerPluWie2000}
Henning {Behnke}, Ulrich {Mertins}, Michael {Plum}, and Christian {Wieners},
  {\em {Eigenvalue inclusions via domain decomposition.}}, {Proc. R. Soc.
  Lond., Ser. A, Math. Phys. Eng. Sci.} {\bf 456} (2000), no.~2003, 2717--2730
  (English).

\bibitem{BerMolSta2017}
Fleurianne Bertrand, Marcel Moldenhauer, and Gerhard Starke, {\em A posteriori
  error estimation for planar linear elasticity by stress reconstruction},
  preprint arXiv:1703.00436v1 (2017), 20p.

\bibitem{Boffi:2010}
Daniele Boffi, {\em Finite element approximation of eigenvalue problems}, Acta
  Numer. {\bf 19} (2010), 1--120.

\bibitem{Braess2013}
Dietrich {Braess}, {\em {Finite Elemente. Theorie, schnelle L\"oser und
  Anwendungen in der Elastizit\"atstheorie.}}, 5th revised ed. ed., Berlin:
  Springer Spektrum, 2013 (German).

\bibitem{BraSch:2008}
Dietrich Braess and Joachim Sch{\"o}berl, {\em Equilibrated residual error
  estimator for edge elements}, Math. Comp. {\bf 77} (2008), no.~262, 651--672.

\bibitem{CanDusMadStaVoh2017}
Eric Canc{\`e}s, Genevi{\`e}ve Dusson, Yvon Maday, Benjamin Stamm, and Martin
  Vohral{\'\i}k, {\em Guaranteed and robust a posteriori bounds for {L}aplace
  eigenvalues and eigenvectors: conforming approximations}, accepted in SIAM J.
  Numer. Anal., preprint hal-01194364v3 (2017), 25p.

\bibitem{CarGal2014}
Carsten Carstensen and Dietmar Gallistl, {\em Guaranteed lower eigenvalue
  bounds for the biharmonic equation}, Numer. Math. {\bf 126} (2014), no.~1,
  33--51.

\bibitem{CarGed2014}
Carsten Carstensen and Joscha Gedicke, {\em Guaranteed lower bounds for
  eigenvalues}, Math. Comp. {\bf 83} (2014), no.~290, 2605--2629.

\bibitem{DolErnVoh2016}
V{\'\i}t Dolej{\v{s}}{\'\i}, Alexandre Ern, and Martin Vohral{\'\i}k, {\em
  {$hp$}-adaptation driven by polynomial-degree-robust a posteriori error
  estimates for elliptic problems}, SIAM J. Sci. Comput. {\bf 38} (2016),
  no.~5, A3220--A3246.

\bibitem{ErnVoh2013}
Alexandre Ern and Martin Vohral{\'{\i}}k, {\em Adaptive inexact {N}ewton
  methods with a posteriori stopping criteria for nonlinear diffusion {PDE}s},
  SIAM J. Sci. Comput. {\bf 35} (2013), no.~4, A1761--A1791.

\bibitem{GoeHau1985}
F.~Goerisch and H.~Haunhorst, {\em Eigenwertschranken f\"ur {E}igenwertaufgaben
  mit partiellen {D}ifferentialgleichungen}, Z. Angew. Math. Mech. {\bf 65}
  (1985), no.~3, 129--135.

\bibitem{GruOva2009}
Luka Grubi{\v{s}}i{\'c} and Jeffrey~S. Ovall, {\em On estimators for
  eigenvalue/eigenvector approximations}, Math. Comp. {\bf 78} (2009), no.~266,
  739--770.

\bibitem{HasHla:1976}
Jaroslav Haslinger and Ivan Hlav{\'a}{\v{c}}ek, {\em Convergence of a finite
  element method based on the dual variational formulation}, Apl. Mat. {\bf 21}
  (1976), no.~1, 43--65.

\bibitem{Hla:1978}
Ivan Hlav{\'a}{\v{c}}ek, {\em Some equilibrium and mixed models in the finite
  element method}, Mathematical models and numerical methods ({P}apers, {F}ifth
  {S}emester, {S}tefan {B}anach {I}nternat. {M}ath. {C}enter, {W}arsaw, 1975)
  (Warsaw), Banach Center Publ., vol.~3, PWN, Warsaw, 1978, pp.~147--165.

\bibitem{HlaKri:1984}
Ivan Hlav{\'a}{\v{c}}ek and Michal K{\v{r}}{\'{\i}}{\v{z}}ek, {\em Internal
  finite element approximations in the dual variational method for second order
  elliptic problems with curved boundaries}, Apl. Mat. {\bf 29} (1984), no.~1,
  52--69.

\bibitem{HuHuaLin2014}
Jun Hu, Yunqing Huang, and Qun Lin, {\em Lower bounds for eigenvalues of
  elliptic operators: by nonconforming finite element methods}, J. Sci. Comput.
  {\bf 61} (2014), no.~1, 196--221.

\bibitem{HuHuaShe2015}
Jun {Hu}, Yunqing {Huang}, and Quan {Shen}, {\em {Constructing both lower and
  upper bounds for the eigenvalues of elliptic operators by nonconforming
  finite element methods.}}, {Numer. Math.} {\bf 131} (2015), no.~2, 273--302
  (English).

\bibitem{Kato1949}
Tosio Kato, {\em On the upper and lower bounds of eigenvalues}, J. Phys. Soc.
  Japan {\bf 4} (1949), 334--339.

\bibitem{KuzRep2013}
Yu.~A. Kuznetsov and S.~I. Repin, {\em Guaranteed lower bounds of the smallest
  eigenvalues of elliptic differential operators}, J. Numer. Math. {\bf 21}
  (2013), no.~2, 135--156.

\bibitem{Lehmann1949}
N.~Joachim Lehmann, {\em Beitr\"age zur numerischen {L}\"osung linearer
  {E}igenwertprobleme. {I}}, Z. Angew. Math. Mech. {\bf 29} (1949), 341--356.

\bibitem{Lehmann1950}
\bysame, {\em Beitr\"age zur numerischen {L}\"osung linearer
  {E}igenwertprobleme. {II}}, Z. Angew. Math. Mech. {\bf 30} (1950), 1--16.

\bibitem{LiLinXie2013}
Qin Li, Qun Lin, and Hehu Xie, {\em Nonconforming finite element approximations
  of the {S}teklov eigenvalue problem and its lower bound approximations},
  Appl. Math. {\bf 58} (2013), no.~2, 129--151.

\bibitem{Liu2015}
Xuefeng Liu, {\em A framework of verified eigenvalue bounds for self-adjoint
  differential operators}, Appl. Math. Comput. {\bf 267} (2015), 341--355.

\bibitem{LiuOis2013}
Xuefeng Liu and Shin'ichi Oishi, {\em Verified eigenvalue evaluation for the
  {L}aplacian over polygonal domains of arbitrary shape}, SIAM J. Numer. Anal.
  {\bf 51} (2013), no.~3, 1634--1654.

\bibitem{LuoLinXie:2012}
Fusheng Luo, Qun Lin, and Hehu Xie, {\em Computing the lower and upper bounds
  of {L}aplace eigenvalue problem: by combining conforming and nonconforming
  finite element methods}, Sci. China Math. {\bf 55} (2012), no.~5, 1069--1082.

\bibitem{Plum1990}
Michael {Plum}, {\em {Eigenvalue inclusions for second-order ordinary
  differential operators by a numerical homotopy method.}}, {Z. Angew. Math.
  Phys.} {\bf 41} (1990), no.~2, 205--226 (English).

\bibitem{Plum1991}
Michael Plum, {\em Bounds for eigenvalues of second-order elliptic differential
  operators}, Z. Angew. Math. Phys. {\bf 42} (1991), no.~6, 848--863.

\bibitem{Temple1928}
G.~Temple, {\em The theory of {R}ayleigh's principle as applied to continuous
  systems}, Proc. Roy. Soc. London Ser. A {\bf 119} (1928), no.~2, 276--293.

\bibitem{TreBet2006}
Lloyd~N. Trefethen and Timo Betcke, {\em Computed eigenmodes of planar
  regions}, Recent advances in differential equations and mathematical physics,
  Contemp. Math., vol. 412, Amer. Math. Soc., Providence, RI, 2006,
  pp.~297--314.

\bibitem{Complement:2010}
Tom{\'a}{\v s} Vejchodsk{\'y}, {\em Complementarity based a~posteriori error
  estimates and their properties}, Math. Comput. Simulation {\bf 82} (2012),
  no.~10, 2033--2046 (English).

\bibitem{systemaee:2010}
\bysame, {\em Complementary error bounds for elliptic systems and
  applications}, Appl. Math. Comput. {\bf 219} (2013), no.~13, 7194--7205
  (English).

\bibitem{VejSeb2017}
Tom{\'a}{\v{s}} Vejchodsk{\'y} and Ivana {\v{S}}ebestov{\'a}, {\em New
  guaranteed lower bounds on eigenvalues by conforming finite elements},
  preprint arXiv:1705.10180 (2017), 26p.

\bibitem{Weinstein1937}
Alexandre Weinstein, {\em \'etude des spectres des \'equations aux d\'eriv\'ees
  partielles de la th\'eorie des plaques \'elastiques}, (Mem. Sci. Math. 88)
  Paris: Gauthier-Villars, 1937.

\bibitem{YanZhaLin:2010}
Yidu Yang, Zhimin Zhang, and Fubiao Lin, {\em Eigenvalue approximation from
  below using non-conforming finite elements}, Sci. China Math. {\bf 53}
  (2010), no.~1, 137--150.

\end{thebibliography}

\end{document}